\documentclass[a4paper,16pt]{article}
\usepackage[utf8]{inputenc}
\usepackage{amssymb,amsmath,mathrsfs}
\usepackage{amsthm}
\usepackage[colorlinks,
            linkcolor=blue,
            anchorcolor=blue,
            citecolor=green
            ]{hyperref}
\numberwithin{equation}{section}
\usepackage{graphicx}
\usepackage{amsmath}
\newtheorem{theorem}{Theorem}[section]

\theoremstyle{plain}

\newtheorem{prop}[theorem]{Proposition }
\newtheorem{cor}[theorem]{Corollary }

\newtheorem{lem}[theorem]{Lemma }

\newtheorem{rem}[theorem] {Remark}

\usepackage{listings}

\begin{document}

\title{Long time behavior of the NLS-Szeg\H{o} equation }
\author{Ruoci Sun\footnote{Laboratoire de Math\'ematiques d’Orsay, Univ. Paris-Sud \uppercase\expandafter{\romannumeral11}, CNRS, Universit\'e Paris-Saclay, F-91405 Orsay, France (ruoci.sun@math.u-psud.fr).}}

\maketitle

\noindent $\mathbf{Abstract}$ \quad  We are interested in the influence of filtering the positive Fourier modes to the integrable non linear Schrödinger equation. Equivalently, we want to study the effect of dispersion added to the cubic Szeg\H{o} equation, leading to the NLS-Szeg\H{o} equation on the circle $\mathbb{S}^1$
\begin{equation*}
i \partial_t u + \epsilon^{\alpha}\partial_x^2 u = \Pi(|u|^2 u),\qquad 0<\epsilon<1, \qquad \alpha\geq 0.
\end{equation*}There are two sets of results in this paper. The first result concerns the long time Sobolev estimates for small data. The second set of results concerns the orbital stability of plane wave solutions. Some instability results are also obtained, leading to the wave turbulence phenomenon. \\

\noindent $\mathbf{Keywords}$ \quad Cubic Schr\"odinger equation, Szeg\H{o} projector, small dispersion, stability, wave turbulence, Birkhoff normal form
\tableofcontents

\section{Introduction}
We consider the NLS-Szeg\H{o} equation defined on the circle $\mathbb{S}^1$
\begin{equation}\label{NLS-Szego}
i \partial_t u + \partial_x^2 u = \Pi(|u|^2 u), \qquad
u(0,\cdot)= u_0.
\end{equation}Here $\Pi: L^2(\mathbb{S}^1) \to L^2(\mathbb{S}^1)$ denotes the orthogonal projector from $L^2(\mathbb{S}^1)$ onto the space of $L^2$ boundary values of holomorphic functions on the unit disc,
\begin{equation*}
\Pi : \sum_{k\in \mathbb{Z}}u_k e^{ikx} \longmapsto \sum_{k\geq 0}u_k e^{ikx}.
\end{equation*}We denote by $L^2_+:=\Pi( L^2(\mathbb{S}^1)) \subset L^2(\mathbb{S}^1)$, $H^s_+:=H^s(\mathbb{S}^1)\bigcap L^2_+$, for all $s \geq 0$, and $C^{\infty}_+:=C^{\infty}(\mathbb{S}^1)\bigcap L^2_+$.
\subsection{Motivation}
\noindent The NLS-Szeg\H{o} equation can be seen as the combination of two completely integrable systems: the defocusing cubic Schrödinger equation 
\begin{equation} \label{cubic dNLS equation}
i\partial_t u + \partial_x^2 u= |u|^2 u, \qquad (t,x)\in \mathbb{R}\times  \mathbb{S}^1,\\
\end{equation}and the cubic Szeg\H{o} equation
\begin{equation}\label{cubic szego equation}
i\partial_t V = \Pi(|V|^2V),\qquad (t,x)\in \mathbb{R}\times  \mathbb{S}^1.
\end{equation}They have both a Lax Pair structure and the action-angle coordinates, which can be used to obtain their explicit formulas with the inversed spectral method(see Zakharov--Shabat $[\mathbf{\ref{ZAKHAROV SHABAT}}]$, Faddeev--Takhtajan $[\mathbf{\ref{Faddeev-Takhtajan}}]$,  Gr\'ebert--Kappeler $[\mathbf{\ref{Kappeler Grebert}}]$, G\'erard $[\mathbf{\ref{Gerard defocusing NLS integrals}}]$, for the NLS equation and G\'erard--Grellier $[\mathbf{\ref{gerardgrellier1}}, \mathbf{\ref{gerardgrellier2}},\mathbf{\ref{Gerard-grellier explicit formula szego equation}},\mathbf{\ref{Gerard grellier book cubic szego equation and hankel operators}}]$ for the cubic Szeg\H{o} equation). However, these two Lax pairs cannot be combined in order to give a Lax pair for $(\ref{NLS-Szego})$. Moreover, the long time behaviors of these two equations are totally different. \\

\noindent The NLS equation $(\ref{cubic dNLS equation})$ has a sequence of conservation laws controlling every Sobolev norms(see Faddeev--Takhtajan $[\mathbf{\ref{Faddeev-Takhtajan}}]$, Gr\'ebert--Kappeler $[\mathbf{\ref{Kappeler Grebert}}]$, G\'erard $[\mathbf{\ref{Gerard defocusing NLS integrals}}]$), so all the solutions are uniformly bounded in every $H^s$ space. Moreover, Gr\'ebert and Kappeler $[\mathbf{\ref{Kappeler Grebert}}]$ have proved the existence of the global Birkhoff coordinates for NLS equation. So the solutions of $(\ref{cubic dNLS equation})$ are actually almost periodic on $\mathbb{R}$ valued into $H^s(\mathbb{S}^1)$.\\

\noindent Compared to $(\ref{cubic dNLS equation})$, the cubic Szeg\H{o} equation, which stands for a non-dispersive model, has both the Lax pair structure and the wave turbulence phenomenon. Its long time behavior is extremely sensible according to the different initial data. P.Gérard and S.Grellier have shown that(in $[\mathbf{\ref{gerardgrellier growth of sobolev norm}}, \mathbf{\ref{Gerard-grellier explicit formula szego equation}}, \mathbf{\ref{Gerard grellier book cubic szego equation and hankel operators}}]$) for a $G_{\delta}$ dense subset of initial data in $C^{\infty}_+$, the solutions may blow up in $H^s$, for every $s> \frac{1}{2}$ with super–polynomial growth on some sequence of times, while they go back to their initial data on another sequence of times tending to infinity. However, all the $H^{\frac{1}{2}}$-solutions are almost periodic. (see also $\mathbf{Theorem}$ $\mathbf{\ref{Chaotic long time behaviour for cubic szego equation}}$).

\begin{rem}
Consider the following equation without the Szeg\H{o} projector $\Pi$ on $\mathbb{S}^1$:
\begin{equation}\label{ODE without szego projector}
\begin{cases}
i\partial_t V = |V|^2 V,\\
V(0,\cdot)=V_0.
\end{cases}
\end{equation}Then $V(t,x)=e^{it |V_0|^2}V_0(x)$ and we have $\|V(t)\|_{H^s} \simeq |t|^s$, for all $s\geq 0$, if $|V_0|$ is not a constant function. Hence, the Szeg\H{o} projector both accelerates the energy transfer to high frequencies, and facilitates the transition to low frequencies for $(\ref{ODE without szego projector})$.
\end{rem}

\noindent One wonders about whether filtering the positive Fourier modes can change the long time Sobolev estimates of the cubic defocusing Schrödinger equation. So we introduce equation $(\ref{NLS-Szego})$. On the other hand, it can also be obtained from the cubic Szeg\H{o} equation by adding the dispersive term $\partial_x^2$ to its linear part. In order to see the gradual change of the dispersion, we add the parameter $\epsilon^{\alpha}$ in front of the Laplacian $\partial_x^2$ to get a more general model, the NLS-Szeg\H{o} equation (with small dispersion):
\begin{equation}\label{NLS-Szego epsilon^alpha}
i \partial_t u + \epsilon^{\alpha}\partial_x^2 u = \Pi(|u |^2 u ), \qquad u (0,\cdot)= u_0,  \qquad 0<\epsilon<1,  \qquad \alpha\geq 0.
\end{equation}Equation $(\ref{NLS-Szego})$ is the special case $\alpha=0$ for $(\ref{NLS-Szego epsilon^alpha})$. \\

\noindent We endow $L^2_+$ with the canonical symplectic form $\omega(u,v)=\mathrm{Im}\int_{\mathbb{S}^1} \frac{u\bar{v}}{2\pi}  
$. Equation $(\ref{NLS-Szego epsilon^alpha})$ has the Hamiltonian formalism with the energy functional
\begin{equation}\label{energy functional of NLS Szego epsilon alpha}
E^{\alpha,\epsilon}(u)=\frac{\epsilon^{\alpha}}{2} \|\partial_x u\|_{L^2}^2 + \frac{1}{4}\|u\|_{L^4}^4, \qquad  u \in H^1_+.
\end{equation}Besides $E^{\alpha,\epsilon}$, equation $(\ref{NLS-Szego epsilon^alpha})$ has two other conservation laws, 
\begin{equation*}
\begin{cases}
Q(u)=\|u\|_{L^2}^2,\\
I(u)=\mathrm{Im}\int_{\mathbb{S}^1}\overline{u} \partial_x u = \|u\|_{\dot{H}^{\frac{1}{2}}}^2,
\end{cases}
\end{equation*}which give the estimate of the solution for low frequencies:
\begin{equation*}
\sup_{t\in \mathbb{R}}\|u(t)\|_{H^s} \leq \|u_0\|_{L^2}^{1-2s}\|u_0\|_{H^{\frac{1}{2}}}^{2s}, \qquad \forall s \in [0, \frac{1}{2}].
\end{equation*}

\noindent Proceeding as in the case of equation $(\ref{cubic dNLS equation})$, one can prove the global existence and uniqueness of the solution of the NLS-Szeg\H{o} equation in high frequency Sobolev spaces, by using the Brezis--Gallou\"et type estimate $[\mathbf{\ref{Brezis Gallouet inequality article}}]$, the Aubin--Lions--Simon theorem (see $\mathbf{Theorem}$ \uppercase\expandafter{\romannumeral2}.5.16 in Boyer--Fabrie $[\mathbf{\ref{Boyer}}]$) and the Trudinger type inequality (see Yudovich $[\mathbf{\ref{Yudovich trudinger}}]$, Vladimirov $[\mathbf{\ref{Vladimirov trudinger}}]$, Ogawa $[\mathbf{\ref{Ogawa trudinger}}]$ and G\'erard--Grellier $[\mathbf{\ref{gerardgrellier1}}]$). Its well-posedness problem in low frequency Sobolev spaces can be dealt with Strichartz's inequality introduced in Bourgain $[\mathbf{\ref{bourgain strichartz inequality}}]$. Only the high frequency Sobolev estimates are considered in this paper.
\begin{prop}\label{GWP for s bigger than 1/2}
For every $s \geq \frac{1}{2}$, given $u_0\in H^s_+$, there exists a unique solution $u\in C(\mathbb{R},H^s_+)$ of $(\ref{NLS-Szego epsilon^alpha})$ such that $u(0)=u_0$. For every $T>0$, the mapping $u_0 \in H^s_+ \mapsto u \in C([-T,T], H^s_+)$ is continuous.
\end{prop}

\subsection{Main results}
The first result concerns the long time stability around the null solution of the NLS-Szeg\H{o} equation $(\ref{NLS-Szego epsilon^alpha})$. If the initial data $u_0$ is bounded by $\epsilon$, we look for a time interval $I^{\alpha}_{\epsilon}$, in which the solution $u(t)$ is still bounded by $\mathcal{O}(\epsilon)$. Now we state the first result of this paper.
\begin{theorem}\label{orbital stability epsilon^(4-alpha) estimates}
For every $s >\frac{1}{2}$, there exist two constants $a_s \in (0,1)$ and $K_s >0$ such that for all $0<\epsilon \ll 1$ and $u_0 \in H^s_+$, if $\|u_0\|_{H^s} =\epsilon$ and $u$ denotes the solution of $(\ref{NLS-Szego epsilon^alpha})$, then
\begin{equation}\label{epsilon^(4-alpha) estimates formula}
\begin{cases}
\sup_{|t| \leq \frac{a_s}{\epsilon^{4-\alpha}}} \|u(t)\|_{H^s} \leq  K_s \epsilon ,\qquad \mathrm{if} \quad \alpha\in [0,2];\\
\sup_{|t| \leq \frac{a_s}{\epsilon^{2}}} \|u(t)\|_{H^s} \leq K_s\epsilon, \qquad \qquad\mathrm{if} \quad \alpha >2.
\end{cases}
\end{equation}Moreover, the time interval $I^{\alpha}_{\epsilon}=[-\frac{a_s}{\epsilon^{2}},\frac{a_s}{\epsilon^{2}}]$ is maximal for the case $\alpha >2$ and $s \geq 1$ in the following sense: for every $0<\epsilon \ll 1$, there exists $u_0^{\epsilon}\in C^{\infty}_+$ such that $\|u_0^{\epsilon}\|_{H^s}\simeq \epsilon$ and for every $\beta>0$, we have 
\begin{equation*}
\sup_{|t|\leq \frac{1}{\epsilon^{2+\beta}}} \|u(t)\|_{H^s}\gtrsim \epsilon |\ln \epsilon|^{\frac{1}{2}} \gg \epsilon,\qquad \qquad  u(0)=u_0^{\epsilon}.
\end{equation*}
\end{theorem}

\begin{rem}
In the case $\alpha \in [0,2)$, the proof is based on the Birkhoff normal form method, similarly to Bambusi $[\mathbf{\ref{Bambusi birkhoff normal form for nonlinear pde}}]$, Gr\'ebert $[\mathbf{\ref{GB birkoff normal form}}]$, G\'erard--Grellier $[\mathbf{\ref{gerardgrellier effective dynamic}}]$ and Faou--Gauckler--Lubich $[\mathbf{\ref{faou--Gauckler--lubich Sobolev Stability of Plane Wave}}]$ for instance. However, the time interval $[-\frac{a_s}{\epsilon^{4-\alpha}}, \frac{a_s}{\epsilon^{4-\alpha}}]$ may not be optimal. The resonant term of $6$ indices in the homological equation can not be cancelled by the Birkhoff normal form transform.(see $\mathbf{subsubsection}$ $\mathbf{\ref{optimal interval for alpha =0}}$)
\end{rem}

\noindent The second set of results concerns the long time $H^s$-estimates for the solutions of $(\ref{NLS-Szego epsilon^alpha})$, if its initial data is a perturbation of the plane wave $\mathbf{e}_m : x  \mapsto e^{imx}$, for some $m \in  \mathbb{N}$ and $s \geq 1$. Let $u=u(t,x)$ be the solution of equation $(\ref{NLS-Szego epsilon^alpha})$ such that $\|u(0)-\mathbf{e}_m\|_{H^s} = \epsilon$. Its energy functional $(\ref{energy functional of NLS Szego epsilon alpha})$ gives the following estimate:
\begin{equation}\label{estimate of norm H1 de u by energy functional}
\sup_{t\in \mathbb{R}}\|u(t)\|_{H^1}\lesssim_{\|u_0\|_{H^1}}\epsilon^{-\frac{\alpha}{2}}, \qquad \forall 0<\epsilon<1, \quad \alpha\geq 0.
\end{equation}However, no information on the stability of the plane waves $\mathbf{e}_m$ is obtained from $(\ref{estimate of norm H1 de u by energy functional})$ during the process $\epsilon \to 0^+$. Consider the super-polynomial growth of Sobolev norms in the cubic Szeg\H{o} equation case (see G\'erard--Grellier $[\mathbf{\ref{gerardgrellier growth of sobolev norm}},\mathbf{\ref{Gerard grellier book cubic szego equation and hankel operators}}]$ and $\mathbf{Proposition}$ $\mathbf{\ref{daisy effect}}$ in this paper), the occurence of wave turbulence phenomenon for $(\ref{NLS-Szego epsilon^alpha})$ depends on the level of its dispersion. We begin with three long time stability results for the polynomial dispersion $\epsilon^{\alpha}\partial_x^2$ case with $0\leq \alpha\leq 2$. The following theorem indicates $H^1$-orbital stability of the traveling waves $\mathbf{e}_m$ for equation $(\ref{NLS-Szego epsilon^alpha})$.

\begin{theorem}\label{H1 estimates of the difference TRAVELLING WAVE delta}
For all $\epsilon \in (0,1)$, $\alpha \in [0,2]$ and $m \in \mathbb{N}$, there exists $\mathcal{C}_m >0$ such that if $\|u(0)-\mathbf{e}_m\|_{H^1} = \epsilon$, then we have
\begin{equation*}
\sup_{t\in \mathbb{R}} \inf_{\theta\in \mathbb{R}}\|u(t)-e^{i\theta}\mathbf{e}_m\|_{H^1} \leq \mathcal{C}_m \epsilon^{1-\frac{\alpha}{2}}.
\end{equation*}
\end{theorem}

\noindent For each $t \in \mathbb{R}$, the infimum can be attained when $\theta= \arg u_m(t)$. A similar result is established by Zhidkov $[\mathbf{\ref{Zhidov qualitative theory book}}$, Sect. $3.3]$ and Gallay--Haragus $[\mathbf{\ref{Gallay--Haragus stability of small periodic waves}}, \mathbf{\ref{Gallay--Haragus ORBITAL stability of periodic waves}}]$ for the 1D cubic Schr\"odinger equation. In small dispersion case, $\mathbf{Theorem}$ $\mathbf{\ref{H1 estimates of the difference TRAVELLING WAVE delta}}$ gives a significant improvement of estimate $(\ref{estimate of norm H1 de u by energy functional})$. We denote by $S_{\alpha,\epsilon}$ the non linear evolution group defined by $(\ref{NLS-Szego epsilon^alpha})$ on $H^{\frac{1}{2}}_+$. In other words, for every $\phi \in H^{\frac{1}{2}}_+$, $t \mapsto S_{\alpha,\epsilon}(t)\phi$ is the solution $u \in C(\mathbb{R}, H^{\frac{1}{2}}_+)$ of equation $(\ref{NLS-Szego epsilon^alpha})$ such that $u(0)=\phi$.

\begin{cor}
For every $m \in \mathbb{N}$, we have
\begin{equation*}
\sup_{\begin{smallmatrix}
0<\epsilon<1 \\0\leq \alpha\leq 2
\end{smallmatrix}}  \sup_{\|\phi-\mathbf{e}_m\|_{H^1}\leq \epsilon}\sup_{t\in  \mathbb{R}}\|S_{\alpha,\epsilon}(t)\phi\|_{H^1} <\infty.
\end{equation*}
\end{cor}

\noindent Compared to $\mathbf{Proposition}$ $\mathbf{\ref{daisy effect}}$  (see G\'erard--Grellier $[\mathbf{\ref{gerardgrellier1}},\mathbf{\ref{gerardgrellier effective dynamic}},\mathbf{\ref{Gerard-grellier explicit formula szego equation}}]$), the dispersive term $\epsilon^{\alpha}\partial_x^2$ counteracts the wave turbulence phenomenon in $H^1$ norm for equation $(\ref{NLS-Szego epsilon^alpha})$, if $0\leq \alpha\leq 2$. After the change of variable $u(t)=e^{i\arg u_m(t)}(\mathbf{e}_m+\epsilon^{1-\frac{\alpha}{2}}v(t))$, we use a bootstrap argument to get long time orbital stability of the traveling waves $\mathbf{e}_m$ with respect to higher Sobolev norms.

\begin{prop}\label{epsilon ^ (alpha/2 -1) time estimate for NLS Szego epsilon alpha}
For all $s \geq 1$ and $m \in \mathbb{N}$, there exist two constants $b_{m,s} \in (0,1)$ and $L_{m,s}>0$ such that if $0\leq \alpha <2$ and  $\|u(0)-\mathbf{e}_m\|_{H^s}=\epsilon \in (0,1)$, then we have
\begin{equation}\label{bootstrap for orbital stability}
\sup_{|t| \leq \frac{b_{m,s}}{\epsilon^{1-\frac{\alpha}{2}}}}  \inf_{\theta\in \mathbb{R}}\|u(t)-e^{i\theta}\mathbf{e}_m\|_{H^s} \leq L_{m,s}  \epsilon^{1-\frac{\alpha}{2}}.
\end{equation}
\end{prop}

\noindent We also look for a larger time interval in which the estimate $(\ref{bootstrap for orbital stability})$ holds, by using the Birkhoff normal form transformation. But the coefficients in front of the high frequency Fourier modes in the homological equation may be arbitrarily large, if $\alpha \in(0,2)$. For this reason, we return to the case $\alpha=0$ and consider equation $(\ref{NLS-Szego})$. 
\begin{equation*}
i \partial_t u + \partial_x^2 u = \Pi(|u|^2 u), \qquad
u(0,\cdot)= u_0.
\end{equation*}Then the time interval can be enlarged as $[-\frac{d_{m,s}}{\epsilon^2},\frac{d_{m,s}}{\epsilon^2}]$ in this case.

\begin{theorem}\label{effective dynamic of equation NLSF with u0=exp(ix)+epsilon, t leq epsilon^(-2)}
In the case $\alpha=0$, for all $s \geq 1$ and $m \in \mathbb{N}$, there exist three constants $d_{m,s}, \epsilon_{m,s} \in (0,1)$ and $K_{m,s}>0$ such that if $\|u(0)-\mathbf{e}_m\|_{H^s}=\epsilon \in (0, \epsilon_{m,s})$, then we have
\begin{equation*}
\sup_{|t|\leq \frac{d_{m,s}}{\epsilon^2}}\inf_{\theta\in \mathbb{R}}\|u(t)-e^{i\theta}\mathbf{e}_m\|_{H^s} \leq K_{m,s} \epsilon.
\end{equation*}
\end{theorem}

\noindent A similar result is obtained in Faou--Gauckler--Lubich $[\mathbf{\ref{faou--Gauckler--lubich Sobolev Stability of Plane Wave}}]$ for the focusing or defocusing cubic Schrödinger equation on the arbitrarily dimensional torus. (see $\mathbf{Section}$ $\mathbf{\ref{comparison NLS NLSSZEGO}}$ for the comparison between $(\ref{NLS-Szego})$ and $(\ref{cubic dNLS equation})$)\\

\noindent After stating the stability results, we turn to construct a large solution for $(\ref{NLS-Szego epsilon^alpha})$ with respect to the initial data, if the level of dispersion is exponentially small with respect to the level of perturbation of the plane wave $\mathbf{e}_1: x \mapsto e^{ix}$. We state the last result of this paper.

\begin{theorem}\label{Transfer of energy to high frequencies for the NLS Szego equation with very small dispersion}
There exists a constant $K>0$ such that for all $0<\delta \ll 1$, we denote by $U$ the solution of the following NLS-Szeg\H{o} equation with small dispersion
\begin{equation}\label{rescaling NLS Szego equation with initial data exp(ix)+delta}
i\partial_t U + \nu^2 \partial^2_x U = \Pi(|U|^2U), \qquad U(0,x)=e^{ix}+\delta,
\end{equation}where $\nu=e^{-\frac{\pi K}{2 \delta^2}}$, then we have $\|U(t^{\delta})\|_{H^1} \simeq \frac{1}{\delta}$ with $t^{\delta}:=\frac{\pi}{\delta \sqrt{4+\delta^2}}$.
\end{theorem}

\noindent This $H^1$-instability result indicates that the support of the energy functional of equation $(\ref{rescaling NLS Szego equation with initial data exp(ix)+delta})$ is transferred to higher Fourier modes. This phenomenon is similar to the cubic Szeg\H{o} equation case (see G\'erard--Grellier $[\mathbf{\ref{gerardgrellier growth of sobolev norm}}, \mathbf{\ref{Gerard-grellier explicit formula szego equation}}$, $\mathbf{\ref{Gerard grellier book cubic szego equation and hankel operators}}]$) and the 2D cubic NLS case (see Colliander--Keel--Staffilani--Takaoka--Tao $[\mathbf{\ref{Colliander J., Keel M., Staffilani G.,Takaoka H., Tao Transfer of energy }}]$). Compared to $\mathbf{Theorem}$ $\mathbf{\ref{H1 estimates of the difference TRAVELLING WAVE delta}}$, adding the low-level dispersion $e^{-\frac{\pi K}{ \delta^2}} \partial_x^2$ fails to change the quality of wave turbulence phenomenon ($\mathbf{Proposition}$ $\mathbf{\ref{daisy effect}}$) for the cubic Szeg\H{o} equation. \\

\noindent The second part of $\mathbf{Theorem}$ $\mathbf{ \ref{orbital stability epsilon^(4-alpha) estimates}}$ is a consequence of $\mathbf{Theorem}$ $\mathbf{\ref{Transfer of energy to high frequencies for the NLS Szego equation with very small dispersion}}$. Indeed, if $\alpha>2$, we rescale $u(t,x)=\epsilon U(\epsilon^2 t, x)$ with $e^{-\frac{\pi K}{ 2\delta^2}}=\nu = \epsilon^{\frac{\alpha-2}{2}}$. Then $u$ solves $(\ref{NLS-Szego epsilon^alpha})$ with $u(0,x) = \epsilon (e^{ix}+\delta)$ and 
\begin{equation*}
\|u(\frac{t^{\delta}}{\epsilon^2})\|_{H^1}=\epsilon \|U(t^{\delta})\|_{H^1}\simeq \frac{\epsilon}{\delta}\simeq \epsilon \sqrt{(\alpha-2)| \ln \epsilon|} \gg \epsilon,
\end{equation*}while $\frac{t^{\delta}}{\epsilon^2}\simeq \frac{|\ln \epsilon|}{\epsilon^2} \ll \frac{1}{\epsilon^{2+\beta}}$, for all $\beta >0$. However, this method does not work in the critical case $\alpha=2$. If $u$ solves 
\begin{equation*}
i\partial_t u +\epsilon^2 \partial_x^2 u =\Pi(|u|^2 u), \qquad u(0,x)=\epsilon(e^{ix}+\delta),
\end{equation*}after rescaling $U(t,x)=\epsilon^{-1}u(\epsilon^{-2}t, x)$, we get equation $(\ref{rescaling NLS Szego equation with initial data exp(ix)+delta})$ with $\nu=1$, leading to $(\ref{NLS-Szego})$ with initial data $U(0,x)=e^{ix}+\delta$. $\mathbf{Theorem}$ $\mathbf{\ref{H1 estimates of the difference TRAVELLING WAVE delta}}$ and $\mathbf{Theorem}$ $\mathbf{\ref{effective dynamic of equation NLSF with u0=exp(ix)+epsilon, t leq epsilon^(-2)}}$ yield the following two estimates 
\begin{equation*}
\sup_{t\in \mathbb{R}}\|u(t)\|_{H^1} = \mathcal{O}(\epsilon) , \qquad \sup_{|t|\leq \frac{d_{1,s}}{\epsilon^2\delta^2}}\|u(t)\|_{H^s} = \mathcal{O}(\epsilon), \qquad \forall 0 < \delta \ll 1, \quad \forall 0 <\epsilon<1,
\end{equation*}for every $s > \frac{1}{2}$. The problem of the optimal time interval in the case $\alpha=2$ of $\mathbf{Theorem}$ $\mathbf{\ref{orbital stability epsilon^(4-alpha) estimates}}$ remains open.\\

\noindent This paper is organized as follows. In $\mathbf{Section}$ $\mathbf{\ref{Szego equation recall}}$, we recall some basic facts of the cubic Szeg\H{o} equation and its consequences. In $\mathbf{Section}$ $\mathbf{\ref{section small data}}$, we study long time behavior for $(\ref{NLS-Szego epsilon^alpha})$ with small data and prove $\mathbf{Theorem}$ $\mathbf{\ref{Transfer of energy to high frequencies for the NLS Szego equation with very small dispersion}}$ and $\mathbf{Theorem}$ $\mathbf{\ref{orbital stability epsilon^(4-alpha) estimates}}$. In $\mathbf{Section}$ $\mathbf{\ref{Orbital stability section}}$, we study  the orbital stability of the plane waves $\mathbf{e}_m$ for $(\ref{NLS-Szego epsilon^alpha})$ for every $m \in \mathbb{N}$ and give the proof of $\mathbf{Theorem}$ $\mathbf{\ref{H1 estimates of the difference TRAVELLING WAVE delta}}$, $\mathbf{Proposition}$ $\mathbf{\ref{epsilon ^ (alpha/2 -1) time estimate for NLS Szego epsilon alpha}}$ and $\mathbf{Theorem}$ $\mathbf{\ref{effective dynamic of equation NLSF with u0=exp(ix)+epsilon, t leq epsilon^(-2)}}$. We compare the NLS equation and the NLS-Szeg\H{o} equation in $\mathbf{Section}$ $\mathbf{\ref{comparison NLS NLSSZEGO}}$.

\begin{center}
$Acknowledgments$
\end{center}
The author would like to express his gratitude towards Patrick Gérard for his deep insight, generous advice and continuous encouragement. He also would like to thank Jean-Marc Delort, Benoit Grébert, Sandrine Grellier and Thomas Kappeler for useful discussions.

\section{The cubic Szeg\H{o} equation}\label{Szego equation recall}
In this section, we recall some results of the cubic Szeg\H{o} equation
\begin{equation}\label{cubic szego equation appendix}
i \partial_t V= \Pi( |V|^2 V), \qquad V(0,\cdot)=V_0.
\end{equation}

\subsection{The Lax pair structure}
Given $V\in H^{\frac{1}{2}}_+$, the Hankel operator $H_V : L^2_+ \to L^2_+$ is defined by 
\begin{equation*}
H_V(h) = \Pi(V\bar{h}).
\end{equation*}Given $b \in L^{\infty}(\mathbb{S}^1)$, the Toeplitz operator $T_b :L^2_+ \to L^2_+$ is defined by 
\begin{equation*}
T_b(h) = \Pi(b h).
\end{equation*}
\begin{theorem}(G\'erard--Grellier $[\mathbf{\ref{gerardgrellier1}}]$)
Set $V \in C(\mathbb{R};H^s_+)$ for some $s> \frac{1}{2}$. Then $V$ solves the cubic Szeg\H{o} equation if and only if $H_V$ satisfies the following evolutive equation
\begin{equation}\label{Lax equation for Szego}
\partial_t H_V = [B_V,H_V].
\end{equation}where $B_V:=\frac{i}{2} H_V^2-iT_{|V|^2}$. In other words, $(L_V, B_V)$ is a Lax pair for the cubic Szeg\H{o} equation.
\end{theorem}

\noindent The equation $(\ref{Lax equation for Szego})$ yields that the spectrum of the Hankel operator $H_V$ is invariant under the flow of the cubic Szeg\H{o} equation. Thus the quantity $\mathrm{Tr}|H_v|$ is conserved. A theorem of Peller ($[\mathbf{\ref{Peller Hankel operators of class Sp}}]$ Theorem 2 p. 454) states that 
\begin{equation*}
\|V\|_{B^1_{1,1}} \simeq \mathrm{Tr}|H_V|.
\end{equation*}Using the embedding theorem $H^s \hookrightarrow B^1_{1,1} \hookrightarrow L^{\infty} $, for any $s >1$, we have the following $L^{\infty}$ estimate of the Szeg\H{o} flow.
\begin{cor}(G\'erard--Grellier $[\mathbf{\ref{gerardgrellier1}}]$)\label{L infinity estimate}
Assume $V_0 \in H^s_+$ for some $s>1$, then we have 
\begin{equation*}
\sup_{t \in \mathbb{R}}\|V(t)\|_{L^{\infty}} \lesssim_s \|V_0\|_{H^s}
\end{equation*}
\end{cor}

\subsection{Wave turbulence}
The following theorem indicates its chaotic long time behavior with turbulence phenomenon for general initial data.
\begin{theorem}(G\'erard--Grellier $[\mathbf{\ref{gerardgrellier growth of sobolev norm}},\mathbf{\ref{Gerard grellier book cubic szego equation and hankel operators}}]$)\label{Chaotic long time behaviour for cubic szego equation}
1.There exists a $G_{\delta}-$dense set $U \subset C^{\infty}(\mathbb{S}^1)\bigcap L^2_+$ such that if $V_0 \in U$, then there exist two sequences $(t_n)_{n\in \mathbb{N}}$ and $(t_n')_{n\in \mathbb{N}}$ tending to infinity such that
\begin{equation*}
\begin{cases}
\lim_{n \to +\infty} \frac{\|V(t_n)\|_{H^s}}{|t_n|^p}=+\infty, \qquad \forall s > \frac{1}{2},\quad\forall p \geq 1,\\
\lim_{n \to +\infty}V(t_n') = V_0.\\
\end{cases}
\end{equation*}\\
\noindent 2.For every $V_0 \in H^{\frac{1}{2}}_+$, the mapping $t \in \mathbb{R} \mapsto V(t) \in H^{\frac{1}{2}}_+$ is almost periodic. 
\end{theorem}

\subsection{A special case}
\noindent In order to prove the optimality of the case $\alpha>2$ of $\mathbf{Theorem}$ $\mathbf{\ref{orbital stability epsilon^(4-alpha) estimates}}$, we compare the solution $u$ of the NLS-Szeg\H{o} with small dispersion to the solution of the cubic Szeg\H{o} equation with some special initial data. Set $V_0=V^{\delta}_0:=\delta +e^{ix}$, we denote $V^{\delta}$ the solution of $(\ref{cubic szego equation appendix})$. Refering to G\'erard--Grellier $[\mathbf{\ref{gerardgrellier1}}$ Sect. 6.1, 6.2; $\mathbf{\ref{gerardgrellier effective dynamic}}$ Sect. 3; $\mathbf{\ref{Gerard-grellier explicit formula szego equation}}$ Sect. $4]$, we have the following explicit formula
\begin{equation}\label{fractional explicit formula of cubic szego equation daisy effect}
V^{\delta}(t,x)= \frac{a^{\delta}(t)e^{ix}+b^{\delta}(t)}{1-p^{\delta}(t)e^{ix}},
\end{equation}where
\begin{equation*}
\begin{split}
& a^{\delta}(t)=e^{-it(1+\delta^2)}, \quad b^{\delta}(t)=e^{-it(1+\frac{\delta^2}{2})}(\delta\cos(\omega t)-i\frac{2+\delta^2}{\sqrt{4+\delta^2}}\sin(\omega t)),\\
& p^{\delta}(t)=-\frac{2i}{\sqrt{4+\delta^2}}\sin(\omega t)e^{\frac{-it\delta^2}{2}}, \quad \omega=\delta \sqrt{1+\frac{\delta^2}{4}}.
\end{split}
\end{equation*}
\begin{prop}(G\'erard--Grellier $[\mathbf{\ref{gerardgrellier1}},\mathbf{\ref{gerardgrellier effective dynamic}},\mathbf{\ref{Gerard-grellier explicit formula szego equation}}]$)\label{daisy effect}
For $0< \delta \ll 1$, set $t^{\delta}:=\frac{\pi}{2 \omega}=\frac{\pi}{\delta \sqrt{4+\delta^2}} \sim \frac{\pi}{2\delta}$. Let $V^{\delta}$ be the solution of $(\ref{cubic szego equation appendix})$ with $V^{\delta}(0,x)=e^{ix}+\delta$, then we have the following estimate 
\begin{equation*}
\|V^{\delta}(t)\|_{H^s}\lesssim_s \|V^{\delta}(t^{\delta})\|_{H^s} \simeq_s \frac{1}{ \delta^{2s-1}}, \qquad \forall t \in [0, t^{\delta}].
\end{equation*}for every $s> \frac{1}{2}$.
\end{prop}

\begin{proof}
Expanding formula $(\ref{fractional explicit formula of cubic szego equation daisy effect})$ as Fourier series, we have
\begin{equation*}
\begin{split}
\|V^{\delta}(t)\|_{H^s}^2 \simeq_s \frac{|a^{\delta}(t)+b^{\delta}(t)p^{\delta}(t)|^2}{(1-|p^{\delta}(t)|^2)^{2s+1}} = \frac{\|V^{\delta}(t)\|_{\dot{H}^{\frac{1}{2}}}^2}{(1-|p^{\delta}(t)|^2)^{2s-1}}
\end{split}
\end{equation*}with
\begin{equation*}
\|V^{\delta}(t)\|_{\dot{H}^{\frac{1}{2}}}^2 = \frac{|a^{\delta}(t)+b^{\delta}(t)p^{\delta}(t)|^2}{(1-|p^{\delta}(t)|^2)^2}=\|V^{\delta}(0)\|_{\dot{H}^{\frac{1}{2}}}^2 =1.
\end{equation*}By the explicit formula of $p^{\delta}$, we have 
\begin{equation*}
\|V^{\delta}(t)\|_{H^s}^2 \simeq_s \left(\frac{4+\delta^2}{4\cos^2(\omega t)+\delta^2}\right)^{2s-1}  \leq \|V^{\delta}(t^{\delta})\|_{H^s}^2 \simeq \frac{C_s}{\delta^{4s-2}}
\end{equation*}with $t^{\delta}:=\frac{\pi}{2 \omega}=\frac{\pi}{\delta \sqrt{4+\delta^2}} $.

\end{proof}

\section{Long time behavior for small data}\label{section small data}
\subsection{The case $\alpha \geq 2$}\label{proof of case alpha >= 2}
For all $s> \frac{1}{2}$, consider the NLS-Szeg\H{o} equation with small dispersion and small data.
\begin{equation}\label{NLS Szego small initial data epsilon alpha}
i \partial_t u + \epsilon^{\alpha}\partial_x^2 u = \Pi(|u |^2 u ), \qquad \|u(0)\|_{H^s}=\epsilon,  \qquad 0<\epsilon<1,  \qquad \alpha\geq 0.
\end{equation}
\noindent At first, we give the proof of the time interval $I^{\alpha}_{\epsilon}=[-\frac{a_s}{\epsilon^2}, \frac{a_s}{\epsilon^2}]$ in the case $\alpha \geq 2$ of $\mathbf{Theorem}$ $\mathbf{\ref{orbital stability epsilon^(4-alpha) estimates}}$, which is based on a bootstrap argument. Then we prove the maximality of $I^{\alpha}_{\epsilon} $.
\subsubsection{The bootstrap argument}
\begin{lem}\label{bootstrap}
Let $a, b, T>0$, $m >1$ and $M: [0,T]\longrightarrow \mathbb{R}_+$ be a continuous function satisfying 
\begin{equation*}
M(\tau) \leq a+bM(\tau)^m, \qquad \mathrm{for} \qquad \mathrm{all}\qquad \tau \in [0,T]
\end{equation*}Assume that $(mb)^{\frac{1}{m-1}} M(0) \leq 1$ and $(mb)^{\frac{1}{m-1}} a \leq \frac{m-1}{m}$. Then 
\begin{equation*}
M(\tau) \leq \frac{m}{m-1}a
\end{equation*}for all $\tau \in [0,T]$.
\end{lem}

\begin{proof}
The function $f_m: z \in \mathbb{R}_+ \longmapsto z-bz^m$ attains its maximum at the critical point $z_c=(mb)^{-\frac{1}{m-1}}$. $f_m(z_c)=\frac{m-1}{m}(mb)^{-\frac{1}{m-1}}$. Since $a \leq \max_{z\geq 0}f_m(z)=f_m(z_c)$, there exists $z_-\leq z_c \leq  z_+$ such that
\begin{equation*}
\{z \geq 0: f_m(z)\leq a\}=[0, z_-]\cup [z_+, +\infty[
\end{equation*}and $f_m(z_{\pm})=a$. Since $f_m(M(\tau)) \leq a$, $\forall 0\leq \tau \leq T$ and $M(0) \leq z_c$, we have $M([0,T]) \subset [0, z_-]$. By the concavity of $f_m$ on $[0, +\infty[$, we have $f_m(z) \geq \frac{f_m(z_c)}{z_c}z$ for all $z\in [0, z_c]$. Consequently, $M(t) \leq z_- \leq \frac{m}{m-1}a$, for all $0 \leq t \leq T$.
\end{proof}

\bigskip
\bigskip

\begin{proof}[Proof of of estimate $(\ref{epsilon^(4-alpha) estimates formula})$ in the case $\alpha >2$ ]
For all $\alpha \geq 0$ and $\epsilon \in (0,1)$ fixed, we rescale $u$ as $u=\epsilon \mu$, equation $(\ref{NLS Szego small initial data epsilon alpha} )$ becomes 
\begin{equation}\label{equation NLS-Szego alpha epsilon after rescaling}
\begin{cases}
i \partial_t \mu + \epsilon^{\alpha}\partial_x^2 \mu = \epsilon^2\Pi(|\mu|^2 \mu),\\
\|\mu(0)\|_{H^s} =1.
\end{cases}
\end{equation}Duhamel's formula of equation $(\ref{equation NLS-Szego alpha epsilon after rescaling})$ gives the following estimate:
 \begin{equation}
 \sup_{0\leq \tau \leq t} \|\mu(\tau)\|_{H^s}  \leq \|\mu (0)\|_{H^s} + C_s \epsilon^2 t \sup_{0\leq \tau \leq t} \|\mu(\tau)\|_{H^s}^3
 \end{equation}Here $C_s$ denotes the Sobolev constant in the inequality $\||\mu|^2 \mu\|_{H^s} \leq C_s \|\mu\|_{H^s}^3$. We choose $a_s=\frac{4}{27C_s}$ and the following estimate holds 
 \begin{equation}\label{estimate for alpha strickly bigger than 2}
\sup_{|t| \leq \frac{a_s}{\epsilon^{2}}} \|u(t)\|_{H^s} \leq \frac{3}{2}\epsilon, \qquad \forall \alpha \geq 0,
 \end{equation}by using $\mathbf{Lemma}$ $\mathbf{\ref{bootstrap}}$ with $m=3$, $T=\frac{a_s}{\epsilon^2}$, $a= M(0)=1$, $b= C_s \epsilon^2 T$ and $M(t)=\sup_{0\leq \tau \leq t} \|\mu(\tau)\|_{H^s}$.\\
 \end{proof}
 \bigskip
 \bigskip

\subsubsection{Optimality of the time interval if $\alpha >2$}
\noindent In order to prove the optimality of $I^{\alpha}_{\epsilon}$ in which estimate $(\ref{estimate for alpha strickly bigger than 2})$ holds, we set $u(0,x)=\epsilon(e^{ix}+\delta)$ and rescale $u(t,x)=\epsilon U(\epsilon^2 t,x)$. Then, we have
\begin{equation*}
i\partial_t U + \nu^2 \partial^2_x U = \Pi(|U|^2U), \qquad U(0,x)=e^{ix}+\delta,
\end{equation*}where $\nu:=\epsilon^{\frac{\alpha-2}{2}}$. Since the optimality is a consequence of $\mathbf{Theorem}$ $\mathbf{\ref{Transfer of energy to high frequencies for the NLS Szego equation with very small dispersion}}$, we prove at first $\mathbf{Theorem}$ $\mathbf{\ref{Transfer of energy to high frequencies for the NLS Szego equation with very small dispersion}}$ by comparing $U$ to the solution of the cubic Szeg\H{o} equation with the same initial data,
\begin{equation*}
i\partial_t V =\Pi(|V|^2 V), \qquad V(0,x)=e^{ix}+\delta.
\end{equation*}

\begin{proof}[Proof of $\mathbf{Theorem}$ $\mathbf{\ref{Transfer of energy to high frequencies for the NLS Szego equation with very small dispersion}}$]
We shall estimate their difference $r(t,x):=U(t,x)-V(t,x)$, which satisfies the following equation
\begin{equation}\label{equation of r}
i\partial_t r +\nu^2 \partial_x^2 r = -\nu^2 \partial_x^2 V + \Pi(V^2 \overline{r}+ 2|V|^2 r)+ Q(r), \qquad r(0)=0,
\end{equation}with $Q(r):=\Pi(\overline{V}r^2+2V|r|^2 + |r|^2 r)$. Thus, we can calculate the derivative of $\|r(t)\|_{H^1}^2$,
\begin{equation*}
\begin{split}
& \partial_t \|r(t)\|_{H^1}^2\\
 =& \partial_t \|r(t)\|_{L^2}^2 + \partial_t \|\partial_x r(t)\|_{L^2}^2 \\
=& 2 \mathrm{Im}\langle i\partial_t r(t), r(t) \rangle_{L^2} + 2 \mathrm{Im}\langle \partial_x (i\partial_t r(t)), \partial_x r(t) \rangle_{L^2}\\
=& 2   \mathrm{Im}\int_{\mathbb{S}^1} \nu^2\partial_x V \partial_x\overline{r} + V^2 \overline{r}^2 -\overline{V}|r|^2 r  \\ 
& \qquad +  2   \mathrm{Im}\int_{\mathbb{S}^1} -\nu^2 \partial_x^3 V \partial_x \overline{r} + V^2 (\partial_x\overline{r})^2 +2 V\partial_x V \overline{r} \partial_x \overline{r} +4  \mathrm{Re}(\overline{V}\partial_x V) r \partial_x \overline{r}\\
& \qquad +  2   \mathrm{Im}\int_{\mathbb{S}^1}\partial_x \overline{V} r^2 \partial_x \overline{r} + 2 \overline{V}r |\partial_x r|^2 + 2\partial_x V |r|^2 \partial_x \overline{r} +4V\mathrm{Re}(\overline{r}\partial_x r) \partial_x \overline{r}+r^2 (\partial_x \overline{r})^2.
\end{split}
\end{equation*}Then, we have
\begin{equation*}
\begin{split}
& \Big|\partial_t \|r(t)\|_{H^1}^2 \Big|\\
\leq & 2 \nu^2 \|\partial_x r\|_{L^2}(\|\partial_x V\|_{L^2}+\|\partial_x^3 V\|_{L^2}) + 2 \|V\|_{L^{ \infty}}^2 \| r\|_{H^1}^2 +2\|V\|_{L^{ \infty}}\|r\|_{L^{ \infty}} \| r\|_{L^2}^2\\
& \quad +12 \|V\|_{L^{ \infty}}\|r\|_{L^{ \infty}} \|\partial_x V\|_{L^2} \|\partial_x r\|_{L^2} + 6 \|r\|_{L^{ \infty}}^2 \|\partial_x V\|_{L^2} \|\partial_x r\|_{L^2}\\
& \quad +12\|V\|_{L^{ \infty}}\|r\|_{L^{ \infty}}  \|\partial_x r\|_{L^2}^2 + 2 \|r\|_{L^{ \infty}}^2 \|\partial_x r\|_{L^2}^2\\
\leq & \nu^2 (2 \|\partial_x r\|_{L^2}^2+\|\partial_x V\|_{L^2}^2+\|\partial_x^3 V\|_{L^2}^2) + 2 \|V\|_{L^{ \infty}}^2 \| r\|_{H^1}^2\\
&\quad +12 \|V\|_{L^{ \infty}} \|\partial_x V\|_{L^2} \|r\|_{L^{ \infty}}\|\partial_x r\|_{L^2} + \mathcal{O}(\|r\|_{H^1}^3).
\end{split}
\end{equation*}The $L^{\infty}$ estimate of $V$ is given by $\mathbf{Corollary}$ $\mathbf{\ref{L infinity estimate}}$ and the $H^s$ estimate of $V$ is given by $\mathbf{Proposition}$ $\mathbf{\ref{daisy effect}}$, for all $s> \frac{1}{2}$. Thus, we have
\begin{equation*}
M_{\infty}:=\sup_{0<\delta <1}\sup_{t\in \mathbb{R}}\|V(t)\|_{L^{\infty}} < +\infty,
\end{equation*}and there exist $C_1, C_3 >0$ such that
\begin{equation*}
\|\partial_x V(t)\|_{L^2}\leq \frac{C_1}{\delta}, \qquad \|\partial_x^3 V(t)\|_{L^2}\leq \frac{C_3}{\delta^5},
\end{equation*}for all $0\leq t \leq t^{\delta}=\frac{\pi}{\delta \sqrt{4+\delta^2}}$. We use a bootstrap argument to deal with the term $\mathcal{O}(\|r\|_{H^1}^3)$. Set
\begin{equation*}
T:=\sup\{t>0: \sup_{0\leq \tau\leq t}\|r(\tau)\|_{H^1}\leq 1\},
\end{equation*}then we have
\begin{equation*}
\sup_{0\leq t\leq T} \|r(t)\|_{L^{\infty}}\leq C\sup_{0\leq t\leq T}\|r(t)\|_{H^1} \leq C,
\end{equation*}where $C$ denotes the Sobolev constant. Consequently, for all $0 \leq t \leq \min(T, t^{\delta})$, we have
\begin{equation*}
\begin{split}
& \Big|\partial_t \|r(t)\|_{H^1}^2 \Big|\\
\leq & \nu^2 (\|\partial_x V\|_{L^2}^2+\|\partial_x^3 V\|_{L^2}^2)\\
& +  \| r\|_{H^1}^2(2 \nu^2 + 2 \|V\|_{L^{ \infty}}^2+ 12C\|V\|_{L^{\infty}}\|\partial_x V\|_{L^2} + 6C\|r\|_{L^{\infty}}\|\partial_x V\|_{L^2}+12 \|V\|_{L^{\infty}}\|r\|_{L^{ \infty}}+ 2\|r\|_{L^{\infty}}^2) \\
\leq & \nu^2 (\frac{C_1^2}{\delta^2}+\frac{C_3^2}{\delta^{10}})+(2 +2 M_{\infty}^2 +12CM_{\infty} +2C^2 +(12C M_{\infty} + 6C^2)\frac{C_1}{\delta})\| r\|_{H^1}^2\\
\leq & K(\frac{\nu^2}{\delta^{10}}+\frac{\|r(t)\|_{H^1}^2}{\delta}),  \qquad \forall 0 < \delta, \nu <1,
\end{split}
\end{equation*}with $K:=\max(C_1^2+C_3^2, 2 +2 M_{\infty}^2 +12CM_{\infty} +2C^2 +(12C M_{\infty} + 6C^2)C_1)$. We set
\begin{equation*}
\nu=\epsilon^{\frac{\alpha-2}{2}}=e^{- \frac{\pi K}{2\delta^2}} \Longleftrightarrow \delta = \sqrt{\tfrac{\pi K}{(\alpha-2)|\ln\epsilon|}}.
\end{equation*}Using Gr\"onwall's inequality, we deduce that
\begin{equation*}
\|r(t)\|_{H^1}^2 \leq \frac{\nu^2}{\delta^9}e^{\frac{\pi K}{2 \delta^2}} = \delta^{-9}e^{-\frac{\pi K}{2 \delta^2}}\ll 1 \ll \delta^{-2}, \qquad \forall 0 \leq t \leq t^{\delta}, \quad \forall 0 <\delta \ll 1.
\end{equation*}Since $\|V(t^{\delta})\|_{H^1} \simeq \frac{1}{\delta}$ by $\mathbf{Theorem}$ $\mathbf{\ref{daisy effect}}$, we have $\|U(t^{\delta})\|_{H^1}= \|V(t^{\delta})+r(t^{\delta})\|_{H^1}\simeq \frac{1}{\delta}$.

\end{proof}

\noindent Fix $\alpha >2$, for every $0<\epsilon \ll 1$, we set \begin{equation*}
\delta=\delta^{\alpha,\epsilon} := \sqrt{\tfrac{\pi K}{(\alpha-2)|\ln\epsilon|}}\ll 1, \qquad\qquad T^{\alpha, \epsilon}:= \tfrac{t^{\delta}}{\epsilon^2}=\tfrac{\pi}{2\epsilon^2 \delta \sqrt{1+\frac{\delta^2}{4}}}\simeq \tfrac{\sqrt{(\alpha-2)|\ln\epsilon|}}{\epsilon^2}.
\end{equation*}Then we have $\|u(T^{\alpha,\epsilon})\|_{H^1} \simeq \epsilon \sqrt{(\alpha-2)| \ln \epsilon|} \gg \epsilon$, while $u(0,x)=\epsilon(e^{ix}+\delta)$. Then the optimality of $I^{\alpha}_{\epsilon}=[-\frac{a_s}{\epsilon^2},\frac{a_s}{\epsilon^2}]$ is obtained.

\subsection{The case $0\leq \alpha < 2$}\label{proof of orbital stability theorem for alpha <2}
We assume at first that $u(0) \in C^{\infty}_+$ so that the energy functional of $(\ref{NLS Szego small initial data epsilon alpha})$ 
\begin{equation*}
E^{\alpha,\epsilon}(u)= \frac{\epsilon^{\alpha}}{2} \|\partial_x u\|_{L^2}^2 + \frac{1}{4}\|u\|_{L^4}^4
\end{equation*}
is well defined. For general initial data $u(0) \in H^s_+$, if $s\in (\frac{1}{2},1)$, we use the density argument $\overline{C^{\infty}_+}=H^s_+$.\\

\noindent We rescale $u(t,x) \mapsto \epsilon^{-\frac{\alpha}{2}} u (-\epsilon^{\alpha}t,x)$, then the equation $(\ref{NLS Szego small initial data epsilon alpha})$ is reduced to the case $\alpha=0$. It suffices to prove the following estimate 
\begin{equation*}
\sup_{|t|\leq \frac{a_s}{\epsilon^4}}\|u(t)\|_{H^s}=\mathcal{O}(\epsilon)
\end{equation*}if $u$ solves $ i \partial_t u + \partial_x^2 u = \Pi(|u|^2 u)$ with $\|u(0)\|_{H^s}=\epsilon$.
\subsubsection{Identifying the resonances}
The study of the resonant set of the NLS-Szeg\H{o} equation is necessary before Birkhoff normal form transformation. We refer to Eliasson--Kuksin $[\mathbf{\ref{eliasson-kuksin kam for nls}}]$ to see the analysis of the resonances for a more general non linear term and KAM theorem for the NLS equation. \\

\noindent We use again the change of variable $u=\epsilon \mu$ and we rewrite Duhamel's formula of $\mu$ with $\eta(t)= \sum_{k\geq 0} \eta_k (t) e^{ikx}:= e^{-it \partial_x^2} \mu(t)$. Then we have
\begin{equation*}
\eta_k(t) = \mu_0(t) -i\epsilon^2 \sum_{k_1-k_2+k_3-k=0}\int_0^t e^{-i \tau (k_1^2-k_2^2+k_3^2-k^2)}\eta_{k_1}(\tau)\overline{\eta}_{k_2}(\tau)\eta_{k_3}(\tau) \mathrm{d} \tau,
\end{equation*}for all $k \geq 0$. Recall the classical identification of the resonant set 
\begin{equation*}
\begin{cases}
k_1-k_2+k_3-k_4=0\\
k_1^2-k_2^2+k_3^2-k_4^2=0.
\end{cases}\Longleftrightarrow
\begin{cases}
k_1=k_2\\
k_3=k_4
\end{cases} \quad \mathrm{or}\quad 
\begin{cases}
k_1=k_4\\
k_2=k_3
\end{cases}.
\end{equation*}

\noindent In order to cancel all the resonances, we apply the transformation $v(t):=e^{2it \epsilon^2 \|\mu(0)\|_{L^2}^2}\mu(t)$. As $\|\mu\|_{L^2}$ is a conservation law, we have
\begin{equation}\label{NLSF scaling kill resonance}
i\partial_t v(t) +\partial_x^2 v(t)= \epsilon^2 \left[\Pi (|v(t)|^2v(t))-2\|v(t)\|^2_{L^2}v(t)\right], \qquad \forall t \in \mathbb{R}.
\end{equation}The equation $(\ref{NLSF scaling kill resonance})$ can be seen as the Hamiltonian system with respect to the energy function
\begin{equation}\label{Hamiltonian of NLSF scaling kill resonance}
H^{\epsilon}(v)= \frac{1}{2} \|\partial_x v\|_{L^2}^2 + \frac{\epsilon^2}{4}\left( \|v\|_{L^4}^4-2\|v\|_{L^2}^4\right)=: H_0(v)+\epsilon^2 R(v).
\end{equation}Then we have $R(v)=\frac{1}{4}\left(\sum_{\begin{smallmatrix}
k_1-k_2+k_3-k_4=0\\
k_1^2-k_2^2+k_3^2-k_4^2 \ne 0
\end{smallmatrix} }v_{k_1}\overline{v }_{k_2}v_{k_3}\overline{v}_{k_4}-\sum_{k\geq 0}|v_k|^4\right)$.

\subsubsection{The Birkhoff normal form}
Equation $(\ref{NLSF scaling kill resonance})$ is transferred to another Hamiltonian equation which is closer to the linear Schrödinger equation by Birkhoff normal form method. We try to find a symplectomorphism $\Psi_{\epsilon}$ such that the energy functional $H^{\epsilon}$ is reduced to the Hamiltonian 
\begin{equation*}
H^{\epsilon} \circ \Psi_{\epsilon}(v) = H_0(v) + \epsilon^2 \tilde{R}(v) + \mathcal{O}(\epsilon^{4-\alpha}),
\end{equation*}where $\tilde{R}(v)=-\frac{1}{4}\sum_{k\geq 0}|v_k|^4$. $\Psi_{\epsilon}$ is chosen as the value at time $1$ of the Hamiltonian flow of some energy $\epsilon^{2} F$.\\

\noindent We fix the value $s > \frac{1}{2}$. Recall that, given a smooth real valued function $H$, we denote $X_H$ the Hamiltonian vector field, i.e,
\begin{equation*}
\mathrm{d}H(v)(h) = \omega (h,X_H(v)).
\end{equation*}Given two smooth real-valued functions $F$ and $G$ on $H^s_+$, their Poisson bracket $\{F,G\}$ is defined by
\begin{equation}\label{Poisson bracket}
\{F,G\}(v):= \omega (X_F(v),X_G(v))=\frac{2}{i}\sum_{k\geq 0} \left(\partial_{\overline{v}_k}F\partial_{v_k}G - \partial_{\overline{v}_k}G\partial_{v_k}F \right)(v),
\end{equation}for all $v=\sum_{k\geq 0} v_k e^{ikx} \in H^s_+$. In particular, if $F$ and $G$ are respectively homogeneous of order $p$ and $q$, then their Poisson bracket is homogeneous of order $p+q-2$.

\begin{lem}\label{explicit formula of F}
Set $F(v):=\sum_{k_1-k_2+k_3-k_4=0}f_{k_1,k_2,k_3,k_4}v_{k_1}\overline{v}_{k_2}v_{k_3}\overline{v}_{k_4}$, with the coefficients
\begin{equation*}
f_{k_1,k_2,k_3,k_4}=\begin{cases}
\frac{i}{4\left(k_1^2-k_2^2+k_3^2-k_4^2\right)}, \qquad \mathrm{if} \quad k_1^2-k_2^2+k_3^2-k_4^2\ne 0,\\
0, \qquad\qquad\qquad\quad\quad\mathrm{otherwise}.
\end{cases}
\end{equation*}Thus, $F$ is real-valued and its Hamiltonian field $X_F$ is smooth on $H^s_+$ such that $\{F,H_0\}+R=\tilde{R}$ and the following estimates hold.
\begin{equation*}
\begin{cases}
\|X_F(v)\|_{H^s} \lesssim_s \|v\|_{H^s}^3,\\
\|\mathrm{d}X_F(v)\|_{B(H^s)} \lesssim_s \|v\|_{H^s}^2,
\end{cases}
\end{equation*}for all $v \in H^s_+$.
\end{lem}

\begin{proof}
$F$ well defined because $\sup_{(k_1,k_2,k_3,k_4)\in \mathbb{Z}^4} |f_{k_1,k_2,k_3,k_4}| \leq \frac{1}{4}$, the Sobolev embedding yields that
$|F(v)| \leq \frac{1}{4}\left(\sum_{k\geq 0} |\hat{u}(k)|\right)^4 \lesssim_s \|u\|^4_{H^s}$. The Young's convolution inequality $l^1 * l^1 *l^2 \hookrightarrow l^2 $ implies that $\|X_F(v)\|_{H^s} \lesssim_s \|v\|_{H^s}^3$ and $\|\mathrm{d}X_F(v)\|_{B(H^s)} \lesssim_s \|v\|_{H^s}^2$. Using $(\ref{Poisson bracket})$ and the definition of $f_{k_1,k_2,k_3,k_4}$, we have
\begin{equation*}
\begin{split}
\{F,H_0\}(v) =& i\sum_{k_1-k_2+k_3-k_4=0}(k_1^2-k_2^2+k_3^2-k_4^2)f_{k_1,k_2,k_3,k_4}v_{k_1}\overline{v}_{k_2}v_{k_3}\overline{v}_{k_4}\\
= &-\frac{1}{4}\sum_{\begin{smallmatrix}
k_1-k_2+k_3-k_4=0 \\
k_1^2-k_2^2+k_3^2-k_4^2 \ne 0
\end{smallmatrix}}v_{k_1}\overline{v}_{k_2}v_{k_3}\overline{v}_{k_4}\\
=& -R(v)+\tilde{R}(v). \\
\end{split}
\end{equation*}
\end{proof}

\bigskip
\bigskip

\noindent Set $\chi_{\sigma}:=\exp(\epsilon^{2}\sigma X_F)$ the Hamiltonian flow of $\epsilon^{2} F$, i.e.,
\begin{equation*}
\frac{\mathrm{d}}{\mathrm{d}\sigma} \chi_{\sigma}(u)=\epsilon^{2} X_F(\chi_{\sigma}(u)), \qquad \chi_{0}(u)=u.
\end{equation*}We perform the canonical transformation $\Psi_{\epsilon}:=\chi_1 = \exp(\epsilon^{2} X_F)$. The next lemma will prove the local existence of $\chi_{\sigma}$, for $|\sigma|\leq 1$ and give the estimate of the difference between $v$ and $\Psi_{\epsilon}^{-1}(v)$

\begin{lem}\label{estimates of hamiltonian flow of epsilon 2 F}
For $s>\frac{1}{2}$, there exist two constants $\rho_s, C_s>0$ such that for all $v\in H^s_+$, if  $\epsilon  \|v\|_{H^s} \leq\rho_s$, then $\chi_{\sigma}(v)$ is well defined on the interval $[-1,1]$ and the following estimates hold:
\begin{equation*}
\begin{split}
&\sup_{\sigma\in [-1,1]}\|\chi_{\sigma}(v)\|_{H^s} \leq \frac{3}{2}\|v\|_{H^s},\\
&\sup_{\sigma\in [-1,1]}\|\chi_{\sigma}(v)-v\|_{H^s} \leq C_s\epsilon^{2} \|v\|^3_{H^s},\\
&\|\mathrm{d}\chi_{\sigma}(v)\|_{B(H^s)} \leq \exp(C_s\epsilon^{2} \|v\|^2_{H^s}|\sigma|),\qquad \forall \sigma \in [-1,1].
\end{split}
\end{equation*}
\end{lem}
\bigskip

\begin{proof}
The inequality $\|\mathrm{d}X_F(v)\|_{B(H^s)} \leq C_s \|v\|_{H^s}^2$ implies that the Lipschitz coefficient of the mapping $v \longmapsto \epsilon^{2} X_F(v)$ is bounded by $C_s\epsilon^{2} \|v\|_{H^s}^2 \leq C_s \rho_s^2$. If $\rho_s$ is sufficiently small, then the Hamiltionian flow $(\sigma, v) \mapsto \chi_{\sigma}(v)$ exists on the maximal interval $(-\sigma^*,\sigma^*)$, by the Picard-Lindelöf theorem. Assume that $\sigma^* <1$, then $\mathbf{Lemma}$ $\mathbf{\ref{explicit formula of F}}$ and the following integral formula 
\begin{equation}\label{intergal equation for chi}
\chi_{\sigma}(v) = v + \epsilon^{2} \int_0^{\sigma} X_F(\chi_{\tau}(v))\mathrm{d}\tau,\qquad \forall 0 \leq \sigma <\sigma^*.
\end{equation}yield that
\begin{equation*}
\sup_{0 \leq \tau\leq \sigma}\|\chi_{\tau}(v)\|_{H^s} \leq \|v\|_{H^s} + C_s \sigma \epsilon^{2}\sup_{0 \leq \tau\leq \sigma}\|\chi_{\tau}(v)\|_{H^s}^3, \qquad \forall 0 \leq \sigma <\sigma^*<1.
\end{equation*}By $\mathbf{Lemma}$ $\mathbf{\ref{bootstrap}}$ with $M(t)=\sup_{0 \leq \tau\leq t}\|\chi_{\tau}(v)\|_{H^s}$, $m=3$, $a=M(0)=\|v\|_{H^s}$ and $b=C_s\epsilon^{2}$, we have
\begin{equation*}
\sup_{|\sigma|\leq \sigma^*}\|\chi_{\sigma}(v)\|_{H^s} \leq \frac{3}{2}\|v\|_{H^s}, 
\end{equation*}if $\epsilon \|v\|_{H^s} \leq \frac{2}{3\sqrt{3 C_s}}$. This is a contradiction to the blow-up criterion. Hence $\sigma^* \geq 1$, and we have $\sup_{|\sigma|\leq 1}\|\chi_{\sigma}(v)\|_{H^s} \leq \frac{3}{2}\|v\|_{H^s}$, if $\epsilon \|v\|_{H^s} \leq \rho_s:= \frac{2}{3\sqrt{3 C_s}}$.\\

\noindent Since $\|X_F(v)\|_{H^s} \leq C_s \|v\|_{H^s}^3$, for all $\sigma\in [-1,1]$, we have
\begin{equation*}
\|\chi_{\sigma}(v)-v\|_{H^s}\leq |\sigma| \epsilon^{2}\sup_{0\leq t\leq |\sigma|}\|X_F(\chi_t(v))\|_{H^s}\leq C_s\epsilon^2 \sup_{0\leq t\leq |\sigma|}\|\chi_t(v)\|^3_{H^s}\leq C_s\epsilon^{2} \|v\|^3_{H^s}.
\end{equation*}if $\epsilon \|v\|_{H^s} \leq \rho_s$. We differentiate equation $(\ref{intergal equation for chi})$ and use again $\mathbf{Lemma}$ $\mathbf{\ref{explicit formula of F}}$ to obtain
\begin{equation*}
\begin{split}
\|\mathrm{d}\chi_{\sigma}(u)\|_{B(H^s)} = & \|\mathrm{Id}_{H^s} + \epsilon^2 \int_0^{\sigma}\mathrm{d}X_F(\chi_t(u))\mathrm{d}\chi_{t}(u) \mathrm{d}t\|_{B(H^s)}\\
\leq & 1 +\epsilon^2 \Big|\int_0^{\sigma}\|\mathrm{d}X_F(\chi_t(v))\|_{B(H^s)}\|\mathrm{d}\chi_{t}(v)\|_{B(H^s)} \mathrm{d}t\Big|\\
\leq & 1+C_s\epsilon^2 \|v\|_{H^s}^2\Big|\int_0^{\sigma}\|\mathrm{d}\chi_{t}(v)\|_{B(H^s)} \mathrm{d}t\Big|\\
\leq & e^{C_s\epsilon^2 |\sigma|\|v\|_{H^s}^2}, \qquad \forall \sigma \in [-1,1].
\end{split}
\end{equation*}Here we use the Gronwall inequality in the last step.\\
\end{proof}
\bigskip

\noindent Recall that $\Psi_{\epsilon}=\chi_1$. The normal form of the energy $H^{\epsilon}$ is given below.
\begin{lem}\label{Birkhoff normal form transformation}
For $s>\frac{1}{2}$, there exists a smooth mapping $Y : H^s_+ \longrightarrow H^s_+$ and a constant $C'_s >0$ such that
\begin{equation*}
\begin{cases}
X_{H^{\epsilon}\circ \Psi_{\epsilon}}= X_{H_0}+\epsilon^2 X_{\tilde{R}}+\epsilon^{4} Y,\\
\|Y(v)\|_{H^s}\leq C'_s \|v\|_{H^s}^5,
\end{cases}
\end{equation*}for all $v\in H^s_+$ such that $\epsilon \|v\|_{H^s}\leq \rho_s$. Let us set $w(t):=\Psi_{\epsilon}^{-1}(v(t))$, then we have
\begin{equation}\label{estimates of w(t) epsilon^(4-alpha)}
\Big|\frac{\mathrm{d}}{\mathrm{d}t} \|w(t)\|_{H^s}^2 \Big| \leq C'_s \epsilon^{4} \|w(t)\|^6_{H^s}, 
\end{equation}if $\epsilon \|w(t)\|_{H^s}\leq \rho_s$.

\end{lem}

\begin{proof}
We expand the energy $H^{\epsilon}\circ \Psi_{ \epsilon} = H^{ \epsilon}\circ \chi_1$ with Taylor's formula at time $\sigma=1$ around $0$. Since $\chi_0 = \mathrm{Id}_{H^s_+}$, one gets
\begin{equation*}
\begin{split}
H^{\epsilon}\circ \chi_1 =&  H_0\circ \chi_1+\epsilon^2 R\circ \chi_1\\
=& \left( H_0+ \frac{\mathrm{d}}{\mathrm{d}\sigma} [H_0\circ \chi_{\sigma}]\arrowvert_{\sigma=0}+\int_0^1(1-\sigma)\frac{\mathrm{d}^2}{\mathrm{d}\sigma^2}[H_0\circ \chi_{\sigma}]\mathrm{d}\sigma \right) + \epsilon^2 R+\epsilon^2\int_0^1 \frac{\mathrm{d}}{\mathrm{d}\sigma}[R\circ \chi_{\sigma}]\mathrm{d}\sigma\\
= &  H_0 + \epsilon^2 \left[\{F,H_0\} + R\right] +\epsilon^{4}\int_0^1(1-\sigma)\{F,\{F,H_0\}\}\circ \chi_{\sigma} + \{F,R\}\circ\chi_{\sigma} \mathrm{d}\sigma\\
= &   H_0+ \epsilon^2 \tilde{R}+\epsilon^{4}\int_0^1\left[(1-\sigma)\{F, \tilde{R}\}+\sigma\{F,R\} \right]\circ \chi_{\sigma}\mathrm{d}\sigma\\
\end{split}
\end{equation*}We set $G(\sigma):=(1-\sigma)\{F, \tilde{R}\}+\sigma\{F,R\}$, $\forall \sigma\in [0,1]$. Since $X_{\{F,R\}}$ and $X_{\{F,\tilde{R}\}}(u)$ are homogeneous of degree $5$ with uniformly bounded coefficients, we have 
\begin{equation*}
\|X_{G(\sigma)}(v)\|_{H^s} \leq (1-\sigma)\|X_{\{F,\tilde{R}\}}(v)\|_{H^s} +\sigma\|X_{\{F,R\}}(v)\|_{H^s}\lesssim_s \|v\|_{H^s}^5, \quad \forall v\in H^s_+,
\end{equation*}By the chain rule of Hamiltonian vector fields:
\begin{equation}\label{composition symplectic gradient}
X_{G(\sigma)\circ \chi_{\sigma}}(v)= \mathrm{d}\chi_{-\sigma}(\chi_{\sigma}(v))\circ X_{G(\sigma)}( \chi_{\sigma}(v)), \quad \forall v\in H^s_+, \quad \forall \sigma \in [0,1],
\end{equation}and $\mathbf{Lemma}$ $\mathbf{\ref{estimates of hamiltonian flow of epsilon 2 F}}$, we have
\begin{equation*}
\|X_{G(\sigma)\circ \chi_{\sigma}}(v)\|_{H^s} \leq \|\mathrm{d}\chi_{-\sigma}(\chi_{\sigma}(v))\|_{B(H^s)}\|X_{G(\sigma)}(\chi_{\sigma}(v))\|_{H^s} \lesssim_s e^{C_s\epsilon^{2}\|v\|_{H^s}^2}\|\chi_{\sigma}(v)\|_{H^s}^5 \lesssim_s \|v\|_{H^s}^5,
\end{equation*}for all $v\in H^s_+$ such that $\epsilon \|v\|_{H^s}\leq \rho_s$. Thus we define $Y:=\int_0^1X_{G(\sigma)\circ \chi_{\sigma}}\mathrm{d}\sigma$ and we have
\begin{equation*}
X_{H^{\epsilon}\circ \Psi_{\epsilon}}= X_{H_0}+ \epsilon^2 X_{\tilde{R}}+\epsilon^{4} Y.
\end{equation*}If $\epsilon \|v\|_{H^s}\leq \rho_s$, then $\|Y(v)\|_{H^s}\lesssim_s\|v\|_{H^s}^5$. \\

\noindent Since $\widehat{X_{\tilde{R}}(w)}(k)=-i |w_k|^2 w_k$, $\forall k \geq 0$ and $w(t)=\Psi_{\epsilon}^{-1}(v(t))$, we have the following infinite dimensional Hamiltonian system on the Fourier modes:
\begin{equation}\label{v_1 ode form}
i \partial_t w_k(t) - k^2 w_k(t) -\epsilon^2 |w_k(t)|^2w_k(t)= i \epsilon^{4}\widehat{Y(w(t))}(k), \quad \forall k \geq 0.
\end{equation}If $\epsilon \|w(t)\|_{H^s}\leq \rho_s$, then we have
\begin{equation*}
\begin{split}
\Big|\partial_t \|w(t)\|_{H^s}^2 \Big| \leq   2 \epsilon^{4} \|Y(w(t))\|_{H^s}\|w(t)\|_{H^s} \lesssim_s   \epsilon^{4}\|w(t)\|_{H^s}^6.
\end{split}
\end{equation*}

\end{proof}
\bigskip

\subsubsection{End of the proof of the case  $0\leq \alpha < 2$ }
\begin{proof}Recall that $w(t)= \chi_{-1}(v(t))$ and $\|v(0)\|_{H^s}=1$. $\mathbf{Lemma}$ $\mathbf{\ref{estimates of hamiltonian flow of epsilon 2 F}}$ yields that
\begin{equation*}
\|v(0)-w(0)\|_{H^s}=\|v(0)-\chi_{-1}(v(0))\|_{H^s}\leq C_s \epsilon^{2} \|v(0)\|_{H^s}^3 \leq C_s. 
\end{equation*}Set $K_s:=3(C_s+1)$. Then $\|w(0)\|_{H^s} \leq \frac{K_s}{3}$. We define
\begin{equation*}
\epsilon_s:=\min \left(\frac{m_s}{3K_s},\quad \sqrt{8C_s K_s^2} \right)
\end{equation*}and
\begin{equation*}
T:=\sup\{t\geq 0 :  \sup_{0\leq \tau  \leq t}\|v(\tau)\|_{H^s} \leq 2K_s\}.
\end{equation*}For all $\epsilon \in (0, \epsilon_s)$ and $t\in [0,T]$, we have $\epsilon \|v(t)\|_{H^s} \leq m_s$. Hence $\mathbf{Lemma}$ $\mathbf{\ref{estimates of hamiltonian flow of epsilon 2 F}}$ gives the following estimate
\begin{equation*}
\begin{split}
\|w(t)\|_{H^s} \leq & \|v(t)\|_{H^s} + \|\chi_{-1}(v(t))-v(t)\|_{H^s}\\
\leq & \|v(t)\|_{H^s} + C_s \epsilon^{2} \|v(t)\|_{H^s}^3\\
\leq & 2K_s + 8C_s K_s^3 \epsilon^{2} \\
\leq & 3K_s.
\end{split}
\end{equation*}So we have $\epsilon \sup_{0\leq t \leq T} \|w(t)\|_{H^s}  \leq m_s$ and $\Big|\frac{\mathrm{d}}{\mathrm{d}t} \|w(t)\|_{H^s}^2 \Big|\leq C'_s \epsilon^{4} \|w(t)\|^6_{H^s}$, by $\mathbf{Lemma}$ $\mathbf{\ref{Birkhoff normal form transformation}}$. Set $a_s:=\frac{1}{3^7 K_s^4 C'_s}$. We can precise the estimate of $\|w(t)\|_{H^s}$ by limiting $|t|\leq a_s\epsilon^{-4}$:
\begin{equation*}
\begin{split}
\|w(t)\|_{H^s}^2 \leq &\|w(0)\|_{H^s}^2 + C'_s |t| \epsilon^{4} \|w(t)\|^6_{H^s} \leq  \frac{K_s^2}{9} + 3^6 C'_s K_s^6 |t| \epsilon^{4}\leq  \frac{4K_s^2}{9},
\end{split}
\end{equation*}for all $0 \leq t \leq \min(T,  \frac{a_s}{\epsilon^{4}})$. Then we have
\begin{equation*}
\begin{split}
\|v(t)\|_{H^s} \leq  \|w(t)\|_{H^s}+\|\chi_1(w(t))-w(t)\|_{H^s}
 \leq  \|w(t)\|_{H^s}+ C_s \epsilon^{2 } \|w(t)\|_{H^s}^3
 \leq K_s
\end{split}
\end{equation*}for all $t \in [0,  \frac{a_s}{\epsilon^{4}}]$. Consequently, we have
\begin{equation*}
\sup_{0\leq t \leq \frac{a_s}{\epsilon^{4}}}\|u(t)\|_{H^s}=\epsilon \sup_{0\leq t \leq \frac{a_s}{\epsilon^{4}}}\|v(t)\|_{H^s} \leq K_s  \epsilon.
\end{equation*}In the case $t<0$, we use the same procedure and we replace $t$ by $-t$.
\end{proof}

\subsubsection{The open problem of optimality}\label{optimal interval for alpha =0}
\noindent Let $u$ be the solution of the NLS-Szeg\H{o} equation 
\begin{equation}\label{NLS Szego with small data for the higher order time interval}
i \partial_t u + \partial_x^2 u = \Pi(|u|^2 u), \qquad
\|u(0,\cdot)\|_{H^s}=\epsilon.
\end{equation}Recall that $H^{\epsilon}=H_0 +\epsilon^2 R$ is the energy functional of the equation $(\ref{NLS Szego with small data for the higher order time interval})$ with
\begin{equation*}
\begin{split}
H_0(v)=\frac{1}{2} \|\partial_x v\|_{L^2}^2,\qquad R(v)=\frac{1}{4}(\|v\|_{L^4}^4-\|v\|_{L^2}^4).
\end{split}
\end{equation*}$\chi_{\sigma}=\exp(\epsilon^2\sigma X_{ F})$, with $F(v):=\sum_{k_1-k_2+k_3-k_4=0}f_{k_1,k_2,k_3,k_4}v_{k_1}\overline{v}_{k_2}v_{k_3}\overline{v}_{k_4}$, with the coefficients
\begin{equation*}
f_{k_1,k_2,k_3,k_4}=\begin{cases}
\frac{i}{4\left(k_1^2-k_2^2+k_3^2-k_4^2\right)}, \qquad \mathrm{if} \quad k_1^2-k_2^2+k_3^2-k_4^2\ne 0,\\
0, \qquad\qquad\qquad\quad\quad\mathrm{otherwise}.
\end{cases}
\end{equation*}We recall also that $\tilde{R}(v)=\{F,H_0\}(v)+R(v)=-\frac{1}{4}\sum_{k\geq 0}|v_k|^4$.

\noindent In order to get a longer time interval in which the solution is uniformly bounded by $\mathcal{O}(\epsilon)$, we expand the Hamiltonian $H^{\epsilon} \circ {\chi_1}$ by using the Taylor expansion of higher order to see whether the resonances can be cancelled by the Birkhoff normal form method.
\begin{equation*}
\begin{split}
& H^{\epsilon} \circ {\chi_1}\\
 = & H_0\circ{\chi_1} +\epsilon^2 R\circ{\chi_1} \\
=& H_0 + \partial_{\sigma}(H_0 \circ{\chi_{\sigma}})\big|_{{\sigma}=0} + \frac{1}{2}\partial^2_{\sigma}(H_0 \circ{\chi_{\sigma}})\big|_{{\sigma}=0} + \frac{1}{2} \int_0^1(1-\sigma)^2\partial_{\sigma}^3(H_0 \circ{\chi_{\sigma}})\mathrm{d}\sigma\\
& \qquad + \epsilon^2\left(R + \partial_{\sigma}(R \circ{\chi_{\sigma}})\big|_{{\sigma}=0} +\int_{0}^1 (1-\sigma)\partial_{\sigma}^2(H_0 \circ{\chi_{\sigma}})\mathrm{d}\sigma \right)\\
=&H_0 + \epsilon^2 \tilde{R}+ \frac{\epsilon^4}{2}\{F,R+\tilde{R}\} + \frac{\epsilon^6}{2}\int_0^1 (1-\sigma) \{F,\{F,(1-\sigma)\tilde{R}+(1+\sigma)R\}\}\circ\chi_{\sigma}\mathrm{d}\sigma
\end{split}
\end{equation*}We try to cancel the term $\frac{\epsilon^4}{2}\{F,R+\tilde{R}\}$ by using a canonical transform to $H_1^{\epsilon}=H^{\epsilon} \circ {\chi_1}$ with the following functional 
\begin{equation*}
G(v)=\sum_{k_1-k_2+k_3-k_4+k_5-k_6=0}g_{k_1,k_2,k_3,k_4,k_5,k_6}v_{k_1}\overline{v}_{k_2} v_{k_3} \overline{v}_{k_4} v_{k_5} \overline{v}_{k_6}.
\end{equation*}We want to solve the homological equation $\{G, H_0\}+\frac{1}{2}\{F,R+\tilde{R}\} = 0$. We can calculate that
\begin{equation*}
\begin{split}
&\frac{1}{2}\{F,R+\tilde{R}\}(v)\\
= &2 \mathrm{Im}\sum_{ \begin{smallmatrix}
 k_1-k_2+k_3-k_4+k_5-k_6=0\\
 k_i, k:=k_1-k_2+k_3 \geq 0\\
\end{smallmatrix} }f_{k_1,k_2,k_3,k}v_{k_1}\overline{v}_{k_2} v_{k_3} \overline{v}_{k_4} v_{k_5} \overline{v}_{k_6} \\
 &\qquad - 4 \mathrm{Im}\sum_{ \begin{smallmatrix}
 k_1-k_2+k_3-k_4+k_5-k_6=0\\
 k_i, k:=k_1-k_2+k_3 \geq 0, k_5=k_6\\
\end{smallmatrix} }f_{k_1,k_2,k_3,k}v_{k_1}\overline{v}_{k_2} v_{k_3} \overline{v}_{k_4} v_{k_5} \overline{v}_{k_6} \\
&\qquad - 2 \mathrm{Im}\sum_{ \begin{smallmatrix}
 k_1-k_2+k_3-k_4+k_5-k_6=0\\
 k_i, k:=k_1-k_2+k_3 \geq 0, k_4=k_5=k_6\\
\end{smallmatrix} }f_{k_1,k_2,k_3,k}v_{k_1}\overline{v}_{k_2} v_{k_3} \overline{v}_{k_4} v_{k_5} \overline{v}_{k_6} \\
\end{split}
\end{equation*}In the first term of the preceding formula, there is a resonance set $k_1^2-k_2^2+k_3^2-k_4^2+k_5^2-k_6^2=0$ that cannot be cancelled by the other two terms. We can see Gr\'ebert--Thomann $[\mathbf{\ref{Grebert Thomann Resonant dynamics for the quintic NLS}}]$ and Haus--Procesi $[\mathbf{\ref{Haus Procesi beating solution for quintic nls}}]$ for instance to analyse the resonant set for $6$ indices for the quintic NLS equation.
\begin{equation*}
\begin{split}
&\{G, H_0\}(v)\\
=&i\sum_{k_1-k_2+k_3-k_4+k_5-k_6=0}(k_1^2-k_2^2+k_3^2-k_4^2+k_5^2-k_6^2)g_{k_1,k_2,k_3,k_4,k_5,k_6}v_{k_1}\overline{v}_{k_2} v_{k_3} \overline{v}_{k_4} v_{k_5} \overline{v}_{k_6}.
\end{split}
\end{equation*}Thus the resonant subset
\begin{equation*}
\begin{cases}
k_1-k_2+k_3-k_4+k_5-k_6=0\\
k_1^2-k_2^2+k_3^2-k_4^2+k_5^2-k_6^2=0\\
\end{cases}
\end{equation*}should be cancelled before the Birkhoff normal form transform, just like the step $\mu \mapsto v=e^{2it \epsilon^2 \|\mu(0)\|_{L^2}^2}\mu(t)$, which can cancel all the resonances 
\begin{equation*}
\begin{cases}
k_1-k_2+k_3-k_4=0\\
k_1^2-k_2^2+k_3^2-k_4^2=0\\
\end{cases}
\end{equation*}before we do the canonical transformation $H^{\epsilon} \circ \chi_1$. We only know that
$f_{k_1,k_2,k_3,k_1-k_2+k_3}=f_{k_4,k_5,k_6,k_4-k_5+k_6}$ if \begin{equation*}
\begin{cases}
k_1-k_2+k_3-k_4+k_5-k_6=0\\
k_1^2-k_2^2+k_3^2-k_4^2+k_5^2-k_6^2=0.\\
\end{cases}
\end{equation*}This resonant subset contains the case $k_5 \ne k_6$. The optimality of the time interval for the case $0\leq \alpha < 2$ remains open.

\bigskip
\bigskip

\section{Orbital stability of the traveling plane wave $\mathbf{e}_m$}\label{Orbital stability section}
Consider the following NLS-Szeg\H{o} equation  
\begin{equation}\label{NLS-Szego epsilon^alpha orbital stability}
i \partial_t u + \epsilon^{\alpha}\partial_x^2 u = \Pi(|u |^2 u ), \qquad 0<\epsilon<1,  \qquad 0\leq \alpha \leq 2.
\end{equation}We shall prove at first $H^1$-orbital stability of the traveling waves $\mathbf{e}_m$, for all $m \in \mathbb{N}$. Then, we study their long time $H^s$-stability, for all $s \geq 1$.

\subsection{The proof of $\mathbf{Theorem}$ $\mathbf{\ref{H1 estimates of the difference TRAVELLING WAVE delta}}$}
We follow the idea of using conserved quantities mentioned in Gallay--Haragus $[\mathbf{\ref{Gallay--Haragus stability of small periodic waves}}]$ for equation $(\ref{NLS-Szego epsilon^alpha orbital stability})$.
\begin{proof}
For all $m\in \mathbb{N}$, $0<\epsilon<1$ and $0\leq \alpha \leq 2$, we denote $u(0,x)=e^{imx}+\epsilon f(x)$ with $\|f\|_{H^1}\leq 1$. The NLS-Szeg\H{o} equation has three conservation laws:
\begin{equation*}
\begin{cases}
Q(u(t)):=\|u(t)\|_{L^2}^2=\|u(0 )\|_{L^2}^2;\\
P(u(t)):=(Du(t),u(t))=P(u(0));\\
E^{\alpha, \epsilon}(u(t)):=\frac{\epsilon^{\alpha}}{2}\|\partial_x u(t)\|_{L^2}^2+\frac{1}{4}\|u(t)\|_{L^4}^4= E^{\alpha, \epsilon}(u(0)),
\end{cases}
\end{equation*}with $D=-i \partial_x$ and $(u,v):= \mathrm{Re}\int_{\mathbb{S}^1}\bar{u}v$. 
Thus the following quantity is conserved,
\begin{equation*}
\begin{split}
&\frac{\epsilon^{\alpha}}{2}\|Du(t)- m u(t)\|_{L^2}^2 + \frac{1}{4}\||u(t)|^2-1\|_{L^2}^2\\
=& E^{\alpha, \epsilon}(u(t))-\epsilon^{\alpha}m P(u(t))+ \frac{|m|^2\epsilon^{\alpha}-1}{2}Q(u(t)) +  \frac{1}{4}\\
=&\epsilon^2  \int_{\mathbb{S}^1}|\mathrm{Re}f(x)e^{-imx}|^2\mathrm{d}x + \frac{\epsilon^{2+\alpha}}{2}\|Df-mf\|_{L^2}^2 + \epsilon^3 \int_{\mathbb{S}^1} |f(x)|^2\mathrm{Re}(f(x)e^{-imx})\mathrm{d}x+\epsilon^4 \|f\|_{L^4}^4\\
\lesssim_m &  \epsilon                                                                                                                                                                                                                                                                                                                                                                                                       ^2 .
\end{split}
\end{equation*}Then, we have $\sup_{t \in \mathbb{R}}\|Du(t)- m u(t)\|_{L^2} \lesssim_m \epsilon^{1-\frac{\alpha}{2}} $. Recall that $\mathbf{e}_m(x)=e^{imx}$, then the following estimate holds,
\begin{equation*}
\begin{split}
\|u(t )-u_m(t) \mathbf{e}_m\|_{H^1}^2 =& \sum_{n \ne m} (1+n^2) |u_n(t)|^2 \lesssim_m\|Du(t )- m u(t )\|_{L^2}^2 \lesssim_m \epsilon^{2-\alpha}.
\end{split}
\end{equation*}We have
\begin{equation*}
\begin{split}
&\inf_{\theta \in \mathbb{R}}\|u(t )-u_m(0)e^{i \theta} \mathbf{e}_m\|_{H^1}^2\\
=&\|u(t)-u_m(0)e^{i( \arg u_m(t) - \arg u_m(0))} \mathbf{e}_m\|_{H^1}^2\\
=& (1+m^2)\Big||u_m(t)|-|u_m(0)|\Big|^2 + \|u(t)-u_m(t)  \mathbf{e}_m\|_{H^1}^2
\end{split}
\end{equation*}and by the conservation of $\|u(t)\|_{L^2}$, we have
\begin{equation*}
\begin{split}
 \Big||u_m(t)|-|u_m(0)|\Big|^2
\leq & \Big||u_m(t)|^2-|u_m(0)|^2\Big| \\
=& \Big| \sum_{n\ne m}|u_n(0)|^2- \sum_{n\ne m}|u_n(t)|^2\Big| \\
 = & \max(\|u(0)-u_m(0) \mathbf{e}_m\|_{L^2}^2, \quad \|u(t )-u_m(t)  \mathbf{e}_m \|_{L^2}^2)\\
 \lesssim_m & \epsilon^{2-\alpha}.
\end{split}
\end{equation*}Thus $\sup_{t\in \mathbb{R}}\|u(t )-u_m(0)e^{i( \arg u_m(t) - \arg u_m(0))} \mathbf{e}_m\|_{H^1} \lesssim_m \epsilon^{1-\frac{\alpha}{2}}$. The proof can be finished by $u_m(0)=1+ \epsilon f_m = 1+ \mathcal{O}(\epsilon)$.
\end{proof}

\noindent The preceding theorem also holds for the defocusing NLS equation on $\mathbb{T}^d$, for $d=1,2,3$ (in the energy sub-critical case) with $\mathbb{T}=\mathbb{R} \slash \mathbb{Z} \simeq \mathbb{S}^1$.(see Gallay--Haragus $[\mathbf{\ref{Gallay--Haragus stability of small periodic waves}}, \mathbf{\ref{Gallay--Haragus ORBITAL stability of periodic waves}}]$) We refer to Zhidkov $[\mathbf{\ref{Zhidov qualitative theory book}}$ Sect. $3.3]$ for a detailed analysis of the stability of plane waves.

\begin{rem}
Obtaining the estimate $\sup_{t \in \mathbb{R}}\|u(t )-u_m(t) \mathbf{e}_m\|_{L^{\infty}} \lesssim_m \epsilon^{1-\frac{\alpha}{2}}$ by the Sobolev embedding $H^1(\mathbb{S}^1) \hookrightarrow L^{\infty}$, we can also proceed by using the following estimate, which is uniform on $x$ and $t$,
\begin{equation*}
|u(t,x)|^2 -1 = |u_m(t)|^2-1 +\mathcal{O}_m(\epsilon^{1-\frac{\alpha}{2}}).
\end{equation*}Integrating the preceding term with respect to $x$, we have
\begin{equation*}
||u_m(t)|^2-1| \leq \||u(t)|^2-1\|_{L^2 } + \|\mathcal{O}_m(\epsilon^{1-\frac{\alpha}{2}})\|_{L^2 }  \lesssim_m \epsilon^{1-\frac{\alpha}{2}}
\end{equation*}Thus $\sup_{t\in  \mathbb{R}}|u_m(t)|=1+\mathcal{O}_m(\epsilon^{1-\frac{\alpha}{2}})$ and $
u_m(t) = e^{i\arg u_m(t)}+\mathcal{O}_m(\epsilon^{1-\frac{\alpha}{2}})$. Then we have
\begin{equation*}
\sup_{t\in  \mathbb{R}}\|u(t )-e^{i \arg u_m(t)} \mathbf{e}_m\|_{H^1}^2  \lesssim_m \epsilon^{2-\alpha}.
\end{equation*}Recall that if $z=1+\mathcal{O}(\epsilon)$ then $e^{i\arg z} =1+\mathcal{O}(\epsilon)$. 
\end{rem}

\subsection{Long time $H^s$-stability }
\noindent For every $s \geq 1$, we suppose that $\|u(0) - \mathbf{e}_m\|_{H^s}\leq \epsilon$. Thanks to the estimate
\begin{equation*}
\sup_{t\in \mathbb{R}}\|u(t)-e^{i \arg u_m(t )}\mathbf{e}_m\|_{H^1} \lesssim_m \epsilon^{1-\frac{\alpha}{2}} ,
\end{equation*}we change the variable $u \mapsto v=v^{m,\alpha,\epsilon} (t,x)=\sum_{n\geq 0}v_n(t) e^{inx} \in C^{\infty}(\mathbb{R}\times \mathbb{S}^1)$ such that
\begin{equation}\label{change of variable for orbital stability u to v}
u(t,x)=e^{i\arg u_m(t)}(e^{imx}+\epsilon^{1-\frac{\alpha}{2}}v(t,x))
\end{equation}to study $H^s$-stability of plane waves $\mathbf{e}_m$ and we have 
\begin{equation}\label{uniform estimate of v}
I_m:=\sup_{0\leq \alpha\leq 2}\sup_{0<\epsilon<1}\sup_{t\in \mathbb{R}}\|v(t)\|_{H^1}<+\infty.
\end{equation}
\begin{prop}\label{proposition to define theta varphi alpha<<1}
For every $s \geq 1$, $m\in \mathbb{N}$, $\epsilon\in (0,1)$ and $\alpha \in [0,2)$, if $u$ is smooth and solves $(\ref{NLS-Szego epsilon^alpha orbital stability})$ with $u(0,x)=e^{imx}+\epsilon f(x)$ and $\|f\|_{H^s} \leq 1$, $v$ is defined by formula $(\ref{change of variable for orbital stability u to v})$, then we have
\begin{equation*}
\begin{cases}
v_m(t) \in \mathbb{R}, \qquad\forall t \in \mathbb{R},\\
\sup_{t\in \mathbb{R}}|v_m(t)| \lesssim_m \epsilon^{\min(\frac{\alpha}{2}, 1-\frac{\alpha}{2})},\\
\sup_{t\in \mathbb{R}}|\partial_t v_m(t)|\lesssim_m \epsilon^{1-\frac{\alpha}{2}},\\
\|v(0)\|_{H^s} \lesssim_{m,s} \epsilon^{\frac{\alpha}{2}}.
\end{cases}
\end{equation*}Moreover, there exists a smooth function $\varphi=\varphi_m : \mathbb{R} \to \mathbb{R}\slash 2\pi\mathbb{Z} \simeq \mathbb{S}^1$ and $\epsilon^*_{m} \in (0,1)$ such that for every $1<2^-<2 $, we have
\begin{equation*}
\begin{cases}
\arg u_m(t)=-(1+m^2 \epsilon^{\alpha}) t + \epsilon^{\min(1, 2-\alpha)}\varphi(t),\\
 \sup_{0\leq \alpha\leq 2^-}\sup_{0<\epsilon <\epsilon^*_{m}}\sup_{t \in \mathbb{R}}|\varphi'(t)|< + \infty.\\
\end{cases}
\end{equation*}The parameter $v$ satisfies the following equation
\begin{equation}\label{NLSF with u0=epsilon + exp(ix) equation v alpha <<1}
\begin{split}
& i \partial_t v + \epsilon^{\alpha} \partial_x^2 v - H_{e^{2imx}}(v) -(1-m^2\epsilon^{\alpha}+ \epsilon^{\min(1, 2-\alpha)} \varphi'(t)) v\\
 = &\epsilon^{\min(\frac{\alpha}{2},1-\frac{\alpha}{2})}\varphi'(t)e^{imx} +  \epsilon^{1-\frac{\alpha}{2}}\Pi(e^{-imx}v^2 + 2e^{imx} |v|^2) + \epsilon^{2-\alpha} \Pi(|v|^2v),\\
\end{split}
\end{equation}where $H_{e^{2imx}}(v):=\Pi[e^{2imx}\bar{v}]$ denotes the Hankel operator.
\end{prop}

\begin{proof}
Since $u_m(t)=e^{i\arg u_m(t)}(1+ \epsilon^{1-\frac{\alpha}{2}}  v_m(t))$, we have $1+\epsilon^{1-\frac{\alpha}{2}}  v_m(t) = |u_m(t)| \in \mathbb{R}$. So $v_m(t) \in \mathbb{R}$, for all $t \in \mathbb{R}$. By using the conservation law $\|\cdot\|_{L^2}$ and estimate $(\ref{uniform estimate of v})$, we have
\begin{equation*}
\begin{split}
& 1+2\epsilon \mathrm{Re}f_m + \epsilon^2=  \|u(0)\|_{L^2}^2 =  \|u(t)\|_{L^2}^2= 1+2 \epsilon^{1-\frac{\alpha}{2}} v_m(t) + \epsilon^{2-\alpha} \|v(t)\|_{L^2}^2, \\
\end{split}
\end{equation*}which yields that $\sup_{t\in \mathbb{R}}|v_m(t)|\lesssim_m \epsilon^{\min(\frac{\alpha}{2}, 1-\frac{\alpha}{2})}$. Recall that
\begin{equation*}
u(0,x)= \sum_{n\geq 0} u_n(0) e^{inx} = e^{imx}+\epsilon f(x).
\end{equation*}Then we have $u_m(0)=1+\epsilon f_m =1+\mathcal{O}(\epsilon)$ and $|e^{i\arg u_m(0)}-1|\lesssim \epsilon$. Thus we have 
\begin{equation*}
\epsilon^{1-\frac{\alpha}{2}}\|v(0)\|_{H^s} \lesssim (1+m^2)^{\frac{s}{2}}|e^{i\arg u_m(0)} -1| + \epsilon \|f\|_{H^s} \lesssim_{m,s} \epsilon.
\end{equation*}We define $\theta(t):=\mathrm{arg}(u_m(t))$. Combing the following two formulas \begin{equation*}
i\partial_t u + \epsilon^{\alpha}\partial_x^2 u = e^{i\theta(t)} \left[\epsilon^{1-\frac{\alpha}{2}}  (i\partial_t v + \epsilon^{\alpha}\partial_x^2 v-\theta'(t) v)-(m^2\epsilon^{\alpha}+\theta'(t))e^{imx}\right]
\end{equation*}
\begin{equation*}
\Pi[|u|^2 u]=e^{i\theta(t)}\left[e^{imx}+\epsilon^{1-\frac{\alpha}{2}}(2v +\Pi(e^{2imx}\overline{v})) + \epsilon^{2-\alpha} \Pi(e^{-imx}v^2 +2e^{imx}|v|^2) + \epsilon^{3(1-\frac{\alpha}{2})} \Pi(|v|^2v)  \right]
\end{equation*}we obtain that
\begin{equation}\label{equation with epsilon is not divided alpha}
\begin{split}
&\epsilon^{1-\frac{\alpha}{2}}[i\partial_t v + \epsilon^{\alpha}\partial_x^2 v -H_{e^{2imx}}(v) - (2+\theta'(t))v]\\
=&(1+m^2\epsilon^{\alpha}+\theta'(t))e^{imx} + \epsilon^{2-\alpha}\Pi(e^{-imx}v^2 + 2e^{imx}|v|^2)+\epsilon^{3(1-\frac{\alpha}{2})}  \Pi(|v|^2v),
\end{split}
\end{equation}where $H_{e^{2imx}}(v):=\Pi[e^{2imx}\bar{v}]$ denotes the Hankel operator. The Fourier mode $ v_m(t)$ satisfies the following equation 
\begin{equation*}
\begin{split}
&\epsilon^{1-\frac{\alpha}{2}}\left[i\partial_t v_m(t) - m^2 \epsilon^{\alpha} v_m(t)-\overline{v_m(t)} - (2+\theta'(t))v_m(t)\right]\\
 =&1+m^2\epsilon^{\alpha}+\theta'(t) + \epsilon^{2-\alpha}\Pi(e^{-imx}v^2 + 2e^{imx}|v|^2)_m(t)+\epsilon^{3(1-\frac{\alpha}{2})}\Pi(|v|^2v)_m(t).
 \end{split}
\end{equation*}Estimate $(\ref{uniform estimate of v})$ yields that 
\begin{equation*}
\begin{split}
&\sup_{t \in \mathbb{R}}|\epsilon^{2-\alpha}\Pi(e^{-imx}v^2 + 2e^{imx}|v|^2)_m(t)+\epsilon^{3(1-\frac{\alpha}{2})} \Pi(|v|^2v)_m(t)|  \lesssim_m \epsilon^{2-\alpha}.
\end{split}
\end{equation*}Thus, we have
\begin{equation}\label{equation v_1 with O(epsilon2) alpha}
\epsilon^{1- \frac{\alpha}{2}} \left[i\partial_t v_m(t)  - (3+m^2\epsilon^{\alpha}+\theta'(t))v_m(t)\right] =1+m^2\epsilon^{\alpha}+\theta'(t) + \mathcal{O}_m(\epsilon^{2-\alpha}).
\end{equation}The imaginary part and the real part of $( \ref{equation v_1 with O(epsilon2) alpha})$ give respectively the two following estimates:
\begin{equation*}
\begin{split}
&\sup_{t \in \mathbb{R}}|\partial_t v_m(t)|  \lesssim_m \epsilon^{1- \frac{\alpha}{2}};\\
& 1+m^2\epsilon^{\alpha}+\theta'(t)=\frac{-2\epsilon^{1- \frac{\alpha}{2}}v_m(t)+ \mathcal{O}_m(\epsilon^{2-\alpha})}{1+\epsilon^{1- \frac{\alpha}{2}} v_m(t)}=\frac{\mathcal{O}_m(\epsilon^{\min(1, 2-\alpha)})}{1+\mathcal{O}_m(\epsilon^{\min(1, 2-\alpha)})}=\mathcal{O}_m(\epsilon^{\min(1, 2-\alpha)}).
\end{split}
\end{equation*}for all $0<\epsilon\ll 1$. Then we define $\varphi(t):= \frac{(1+m^2\epsilon^{\alpha}) t + \theta(t)}{\epsilon^{\min(1, 2-\alpha)}}$. Consequently, there exists $\epsilon^*_{m} \in (0,1)$ such that
\begin{equation*}
\begin{cases}
\arg u_m(t)=-(1+m^2\epsilon^{\alpha})t + \epsilon^{\min(1, 2-\alpha)}\varphi(t)\\
 \sup_{0\leq \alpha \leq  2^-}\sup_{0<\epsilon< \epsilon^*_{m }}\sup_{t\in \mathbb{R}}|\varphi'(t)| < +\infty.\\
\end{cases}
\end{equation*}We replace $\theta'(t)$ by $-1-m^2\epsilon^{\alpha}+\epsilon^{\min(1, 2-\alpha)}\varphi'(t)$ in $(\ref{equation with epsilon is not divided alpha})$ and we obtain $(\ref{NLSF with u0=epsilon + exp(ix) equation v alpha <<1})$.
\end{proof}

\subsubsection{Proof of $\mathbf{Proposition}$ $ \mathbf{\ref{epsilon ^ (alpha/2 -1) time estimate for NLS Szego epsilon alpha}}$}
\noindent For every $n\in  \mathbb{N}$, we define the projector $\mathbb{P}_n : L^2_+ \to L^2_+$ such that
 \begin{equation*}
 \mathbb{P}_n (\sum_{k\geq 0}v_k e^{ikx}) = \sum_{j= 0}^n v_k e^{ikx}.
 \end{equation*}Now we prove $\mathbf{Proposition}$ $\mathbf{\ref{epsilon ^ (alpha/2 -1) time estimate for NLS Szego epsilon alpha}}$ by using a bootstrap argument.

\begin{proof}
At the beginning, we suppose that $u(0) \in C^{\infty}_+$. In the general case $u(0)\in H^s_+$, the proof can be completed by using the continuity of the flow $u(0) \mapsto u$ from  $H^s_+$ to $C([-\frac{b_{s,m}}{\epsilon^{1-\frac{\alpha}{2}}},\frac{b_{s,m}}{\epsilon^{1-\frac{\alpha}{2}}}], H^s_+)$. We use the same transformation $u \mapsto v$ as $(\ref{change of variable for orbital stability u to v})$. $\mathbf{Proposition}$ $\mathbf{\ref{proposition to define theta varphi alpha<<1}}$ yields that there exists $A_{m,s}\geq 1$ such that $\|v(0)\|_{H^s}\lesssim_{m,s}\epsilon^{ \frac{\alpha}{2}} \leq A_{m,s}$. By using estimate $(\ref{uniform estimate of v})$, we have
\begin{equation*}
\sup_{\epsilon \in (0,1)}\sup_{t\in \mathbb{R}}\|\mathbb{P}_{2m} (v(t))\|_{H^s}\leq (1+4m^2)^{\frac{s}{2} }I_m
\end{equation*}We define that $L_{m,s}:=\max( 2(1+4m^2)^{\frac{s}{2} }I_m, 2A_{m,s}+1)$ and 
\begin{equation*}
T := \sup \{t >0: \sup_{0\leq \tau \leq t}\|v(\tau)\|_{H^s} \leq 2L_{m,s}\}.
\end{equation*}Rewrite equation $(\ref{NLSF with u0=epsilon + exp(ix) equation v alpha <<1})$ on Fourier modes and we have
\begin{equation*}
i\partial_t v_n -(1+(n^2-m^2)\epsilon^{\alpha} +\epsilon^{\min(1, 2-\alpha)} \varphi'(t))v_n = \epsilon^{1-\frac{\alpha}{2}} [Z(v(t))]_n, \qquad \forall n \geq 2m+1,
\end{equation*}with $Z(v)=\sum_{n\geq 0}[Z(v)]_n e^{inx} = \Pi(e^{-imx}v^2 + 2e^{imx} |v|^2) + \epsilon^{1-\frac{\alpha}{2}} \Pi(|v|^2v) $. Then we have
\begin{equation}\label{estimate of Z }
\|Z(v)-\mathbb{P}_{2m}(Z(v))\|_{H^s} \lesssim_{s} \|v\|_{H^s}^2 + \epsilon^{1-\frac{\alpha}{2}}\|v\|_{H^s}^3. 
\end{equation}Then we have
\begin{equation*}
\begin{split}
|\partial_t \|v(t)-\mathbb{P}_{2m}(v(t))\|_{H^s}^2| \leq & 2 \epsilon^{1-\frac{\alpha}{2}} \sum_{n\geq 2m+1} (1+n^2)^s |v_n(t)||[Z(v(t))]_n|\\
\leq & 2 \epsilon^{1-\frac{\alpha}{2}} \|v(t)\|_{H^s}\|Z(v(t))-\mathbb{P}_{2m}( Z(v(t))\|_{H^s}\\
\lesssim_s & \epsilon^{1-\frac{\alpha}{2}}\|v(t)\|_{H^s}^3 + \epsilon^{2-\alpha}\|v(t)\|_{H^s}^4.\\
\end{split}
\end{equation*}For all $t \in [0, T]$, we have
\begin{equation*}
\begin{split}
\|v(t)\|_{H^s}^2 = &\|\mathbb{P}_{2m} v(t)\|_{H^s}^2+\|v(t)-\mathbb{P}_{2m}(v(t))\|_{H^s}^2\\
\leq & (1+4m^2)^{s}I_m^2+ C_s( \epsilon^{1-\frac{\alpha}{2}}\|v(t)\|_{H^s}^3 + \epsilon^{2-\alpha}\|v(t)\|_{H^s}^4)t + \|v(0)\|_{H^s}^2\\
\leq & \frac{1}{4}L_{m,s}^2 + 32C_s L_{m,s}^4 \epsilon^{1-\frac{\alpha}{2}} t + A_{m,s}^2.
\end{split}
\end{equation*}Define $b_{m,s}=\frac{1}{64C_s L_{m,s}^2}$ and we have $\|v(t)\|_{H^s} \leq  L_{m,s} $, for all $t\in [0, \frac{b_{s,m}}{\epsilon^{1-\frac{\alpha}{2}}}]$. The case $t<0$ is similar.
\end{proof}

\subsubsection{Homological equation}
\noindent We try to improve $\mathbf{Proposition}$ $\mathbf{\ref{epsilon ^ (alpha/2 -1) time estimate for NLS Szego epsilon alpha}}$ and get a longer time interval in which the solution $v$ is still bounded by $\mathcal{O}(1)$, by using The Birkhoff normal form method. Recall the symplectic form $\omega(u,v)=\mathrm{Im}\int_{\mathbb{S}^1}u\overline{v}\frac{\mathrm{d}\theta}{2\pi}
$ on the energy space $H^1_+$ and the Poisson bracket for two smooth real-valued functionals $F,G: C^{\infty}_+ \to \mathbb{R}$
\begin{equation*}
\{F,G\}(v)= \frac{2}{i}\sum_{k\geq 0} \left(\partial_{\overline{v}_k}F\partial_{v_k}G - \partial_{\overline{v}_k}G\partial_{v_k}F \right)(v),
\end{equation*}for all $v=\sum_{k\geq 0} v_k e^{ikx} \in C^{\infty}_+$. For all $0\leq \alpha < 2$ and $0<\epsilon\ll 1$, equation $(\ref{NLSF with u0=epsilon + exp(ix) equation v alpha <<1})$ has also the Hamiltonian formalism, which is non autonomous. Its energy functional is
\begin{equation*}
\begin{split}
&\mathcal{H}^{m,\alpha, \epsilon}(t,v)\\
=&\mathcal{H}_0^{m, \alpha,\epsilon }(v) + \epsilon^{1-\frac{ \alpha}{2}}\mathcal{H}_{1 }^m(v)+\frac{\epsilon^{2-\alpha}}{4}\mathcal{N}_{4 }(t,v) + \epsilon^{\min(\frac{\alpha}{2},1-\frac{\alpha}{2})}\varphi'(t)(\mathcal{L}_m(v)+\frac{\epsilon^{1-\frac{ \alpha}{2}}}{2}\mathcal{N}_2(v)),
\end{split}
\end{equation*}with
\begin{equation*}
\begin{cases}
\mathcal{H}_0^{m,\alpha,\epsilon}(v) = \frac{\epsilon^{\alpha}}{2}\|\partial_x v\|_{L^2}^2 + \frac{1-m^2 \epsilon^{\alpha}}{2}\|v\|_{L^2}^2 + \frac{1}{2}\int_{\mathbb{S}^1}\mathrm{Re}(e^{2imx} \overline{v}^2), \\
\mathcal{H}_{1 }^m( v) = \mathrm{Re}(  \int_{\mathbb{S}^1} e^{-imx} |v|^2v),\\
\mathcal{N}_{p}(v)= \|v\|_{L^p}^p , \qquad p=2 \quad \mathrm{or} \quad 4,\\
\mathcal{L}_m(v)=\mathrm{Re}v_m.\\
\end{cases}
\end{equation*}We want to cancel all the high frequencies in the term $\mathcal{H}_{1 }^m( v)$ by composing $\mathcal{H}^{m,\alpha, \epsilon}$ and the Hamiltonian flow of some auxiliary functional $\mathcal{F}_m$. In order to get the appropriate $\mathcal{F}_m$, we need to solve the homological system
\begin{equation*}
\begin{cases}
\{\mathcal{F}_m, \mathcal{H}^{m,\alpha, \epsilon}_0\}(v)+\mathcal{H}_{1 }^m(v) = \mathcal{R}_m(v)\\
\{\mathcal{F}_m,\mathcal{L}_m\}(v)=-\tilde{\mathcal{N}}_2(v):=-\sum_{n\geq 2m+1} |v_n|^2
\end{cases}
\end{equation*}such that $\mathcal{R}_m$ depends only on finitely many Fourier modes of $v$. The remaining coefficient in front of $\epsilon^{1-\frac{\alpha}{2}}$ would be $\mathcal{R}_m+\varphi'(t)  \epsilon^{\min(\frac{\alpha}{2},1-\frac{\alpha}{2})} (-\tilde{\mathcal{N}}_2+\frac{\mathcal{N}_2}{2})$. One can prove the following proposition.(see also $\mathbf{Proposition}$ $\mathbf{\ref{homological equation alpha =0 prop}}$ and $\mathbf{Appendix}$ for the proof in the special case $\alpha=0$)

\begin{prop}\label{homological equation for all alpha}
For every $m \in \mathbb{N}$, we define the following homogenous functional $\mathcal{F}_m$ of degree $3$:
\begin{equation*}
\mathcal{F}_m(v)=\sum_{\begin{smallmatrix}
 j-l+k=m\\ j,k,l \in \mathbb{N}
\end{smallmatrix}}\mathrm{Re}(a_{j,l,k}v_j \overline{v}_l v_k),\qquad \forall v\in \bigcup_{n\geq 0} \mathbb{P}_n( C^{\infty}_+),
\end{equation*}for some $a_{k,l,j}=a_{j,l,k} \in \mathbb{C}$. Then we have the following formula
\begin{equation*}
\{\mathcal{F}_m,\mathcal{H}_0^{m, \alpha,\epsilon}\}(v)+\mathcal{H}^m_{1 }(v) =\mathrm{Re} \left( \mathrm{Reson}^{low}(v) +\mathrm{Reson}^{\geq 2m+1}(v)\right),
\end{equation*}where 
\begin{equation*}
\begin{split}
& \mathrm{Reson}^{low}(v)\\
= & \sum_{0 \leq j,k \leq 2m} c_{j,j+k-m,k}v_{j} \overline{v}_{j+k-m} v_k 
     -\sum_{ \begin{smallmatrix} 
j-l+k=m  \\
j,l,k \in \mathbb{N}, l\leq 2m
\end{smallmatrix}} ia_{j,l,k}v_j v_k \left[(1+(l^2-m^2)\epsilon^{\alpha}) \overline{v}_l+v_{2m-l} \right],
   \end{split}
\end{equation*}for some $c_{j,j+k-m,k}=c_{k,j+k-m,j}  \in \mathbb{C}$ and
\begin{equation*}
\begin{split}
&  \mathrm{Reson}^{\geq 2m+1}(v)\\
=&\sum_{k\geq 2m+1}\sum_{j=0}^{m-1} 2\left(-i \overline{a}_{2m-j,m+k-j,k}+(1+2(m-j)(k-j)\epsilon^{\alpha})i a_{j,k,k+m-j}+1 \right)v_j \overline{v}_k v_{k+m-j}\\
+&\sum_{k\geq 2m+1}\sum_{j=0}^{m-1} 2((1-2(k-m)(m-j)\epsilon^{\alpha}) i a_{2m-j,m+k-j,k}-i \overline{a}_{j,k,k+m-j}+1)v_{2m-j} \overline{v}_{m+k-j} v_k\\
+&\sum_{k\geq 2m+1}\left(2( i a_{m,k,k}-i \overline{a}_{m,k,k}+1)v_{m} |v_k|^2+((1-2(k-m)^2\epsilon^{\alpha})i a_{k,2k-m,k}+1)v_k^2\overline{v}_{2k-m}\right)\\
+&\sum_{k\geq 2m+1}\sum_{j \geq k+1}2 ((1-2(j-m)(k-m)\epsilon^{\alpha})i a_{k,j+k-m,j}+1) v_{k} \overline{v}_{j+k-m} v_j.
\end{split}
\end{equation*}

\end{prop}

\noindent The term $\mathrm{Reson}^{low}$ depends only on the small Fourier modes $v_1, v_2, \cdots,v_{3m}$. We try to find a bounded sequence $(a_{j,l,k})_{j-l+k=m}$ such that $\mathrm{Reson}^{\geq 2m+1} = 0$ in order to cancel the term $\mathcal{H}^m_{1 }$. However, the coefficient $(1-2(j-m)(k-m)\epsilon^{\alpha})$ in front of the parameter $a_{k,j+k-m,j}$ may have an arbitrarily small absolute value if $\alpha >0$. Such sequence does not exist if $ \epsilon^{-\alpha} \in 2\mathbb{N}\bigcap [2(m+1)^2, +\infty)$.\\

\noindent  We suppose that $\epsilon^{-\alpha} \notin \mathbb{Q}$, then $\mathrm{Reson}^{\geq 2m+1} = 0$ is equivalent to a linear system, which has a unique solution
\begin{equation}\label{Solution of linear system of all a _kln}
\begin{cases}
a_{2m-j,m+k-j,k} = a_{k,m+k-j,2m-j} = \frac{i(k-j)}{(m-j)(1-2(k-m)(k-j)\epsilon^{\alpha})}, \quad \forall 0\leq j \leq m-1,\\
a_{j,k,k+m-j}=a_{k+m-j,k,j} = \frac{i(m-k)}{(m-j)(1-2(k-m)(k-j)\epsilon^{\alpha})}, \quad\quad\quad \forall 0\leq j \leq m-1,\\
a_{m,k,k}=a_{k,k,m}=\frac{i}{2},\\
a_{k,k+j-m,j}=a_{j,k+j-m,k} = \frac{i}{1-2(j-m)(k-m)\epsilon^{\alpha}},\quad \forall j \geq 2m+1,
\end{cases}
\end{equation}for all $k \geq 2m+1$. In the case $m=0$, $(\ref{Solution of linear system of all a _kln})$ has only the last two formulas. When $\alpha > 0$, the sequence $(a_{j,l,k})_{j-l+k=m}$ can be arbitrarily large, for $0<\epsilon \ll1$. We suppose that $\epsilon^{\alpha}$ is an irrational algebraic number of degree $d\geq 2$. Then we have the Liouville estimate $[\mathbf{\ref{Liouville diophantine approximation}}]$
\begin{equation*}
|a_{j,j+k-m,k}| \leq \frac{1}{c_{\epsilon,\alpha}} (2(j-m)(k-m))^{d-1}, \qquad \forall j,k \geq 2m+1,
\end{equation*}which loses the regularity of $v$ in the estimate of $X_{F_m}(v)$. It is difficult to find the same kind of estimate for the transcendental numbers, which can preserve the regularity. So we return to the case $\alpha=0$.

\subsection{Long time $H^s$-stability in the case $\alpha=0$}
\noindent For $\alpha=0$ and every $m \in \mathbb{N}$ and $s \geq 1$, assume that $u$ is the smooth solution of the NLS-Szeg\H{o} equation 
\begin{equation*}
i \partial_t u + \partial_x^2 u = \Pi(|u|^2 u), \qquad
u(0,x)= e^{imx}+\epsilon f(x), \qquad \|f\|_{H^s} \leq 1,
\end{equation*}and $u(t,x)=e^{i\arg u_m(t)}(e^{imx}+\epsilon v(t,x))$. Then $v$ is the solution of the following Hamiltonian equation
\begin{equation*}
\begin{split}
& i \partial_t v + \partial_x^2 v - H_{e^{2imx}}(v) -(1-m^2+ \epsilon  \varphi'(t)) v\\
 = & \varphi'(t)e^{imx} +  \epsilon\Pi(e^{-imx}v^2 + 2e^{imx} |v|^2) + \epsilon^{2} \Pi(|v|^2v).\\
\end{split}
\end{equation*}Its energy functional is
\begin{equation*}
\begin{split}
\mathcal{H}^{m,\epsilon}(t,v)=H_0^{m}(v) + \varphi'(t) \mathcal{L}_m(v) + \epsilon (\mathcal{H}_{1}^m(v)+\frac{\varphi'(t)}{2}\mathcal{N}_2(v))+\frac{\epsilon^{2}}{4}\mathcal{N}_{4}(v),
\end{split}
\end{equation*}with
\begin{equation*}
\begin{cases}
\mathcal{H}_0^{m}(v) = \frac{1}{2}\|\partial_x v\|_{L^2}^2 + \frac{1-m^2 }{2}\|v\|_{L^2}^2 + \frac{1}{2}\int_{\mathbb{S}^1}\mathrm{Re}(e^{2imx} \overline{v}^2) \mathrm{d}x, \\
\mathcal{L}_{m}( v)= \mathrm{Re} v_m,\\
\mathcal{H}_{1}^m( v) = \mathrm{Re}(  \int_{\mathbb{S}^1} e^{-imx} |v|^2v)\mathrm{d}x,\\
\mathcal{N}_p(v)= \|v\|_{L^p}^p, \qquad p=2 \quad \mathrm{or} \quad 4.
\end{cases}
\end{equation*}We define $\tilde{\mathcal{N}}_2(v):= \|v-\mathbb{P}_{2m} v\|_{L^2}^2=\sum_{n\geq 2m+1} |v_n|^2$ and the following proposition holds.

\begin{prop}\label{homological equation alpha =0 prop}
For every $s>\frac{1}{2}$ and $m \in \mathbb{N}$, there exists a sequence $(a_{j,l,k})_{j-l+k=m}$ such that $a_{j,l,k}=a_{k,l,j}$, $\sup_{ m \geq 1}\sup_{j-l+k=m}|a_{j,l,k}|=\frac{1}{2}$ and the functional ${\mathcal{F}_m} : H^s_+ \to \mathbb{R}$, defined by
\begin{equation*}
{\mathcal{F}_m}(v)=\sum_{\begin{smallmatrix}
 j-l+k=m\\ j,k,l \in \mathbb{N}
\end{smallmatrix}}\mathrm{Re}(a_{j,l,k}v_j \overline{v}_l v_k),\qquad \forall v\in C^{\infty}_+,
\end{equation*}satisfy that $\{\mathcal{F}_m, \mathcal{L}_{m}\}=-\tilde{\mathcal{N}}_2$ and  $\mathcal{R}_m :=\{\mathcal{F}_m, \mathcal{H}_0^{m}\} +\mathcal{H}^m_1$ is a finite sum of the Fourier modes $v_1, \cdots, v_{3m}$. Moreover, for all $v,h \in  H^s_+$, we have
\begin{equation*}
\|X_{\mathcal{F}_m}(v) \|_{H^s}\lesssim_{m,s} \|v\|_{H^s}^2,  \qquad\|\mathrm{d}X_{\mathcal{F}_m}(v) h\|_{H^s}\lesssim_{m,s} \|v\|_{H^s}\|h\|_{H^s}.
\end{equation*}
\end{prop}

\begin{proof}
For the convenience of the reader, the detailed calculus for $\mathcal{R}_m=\{\mathcal{F}_m, \mathcal{H}_0^{m}\} +\mathcal{H}^m_1$ and formula $(\ref{Solution of linear system of all a _kln})$ in the case $\alpha=0$ are given in $\mathbf{Appendix}$. We define $a_{j,j+k-m,k}=0$, for all $0\leq j,k \leq 2m$ and $a_{n,m+1+n,2m+1}=a_{2m+1,m+1+n,n}=0$, for all $0\leq n \leq m-1$. Combing $\mathbf{Proposition}$ $\mathbf{\ref{homological equation for all alpha}}$ and $(\ref{Solution of linear system of all a _kln})$ with $\alpha=0$, we have
\begin{equation*}
\begin{cases}
\mathrm{Reson}^{\geq 2m+1}(v) \Big|_{\alpha=0}=0, \qquad \mathrm{Re} \left(\mathrm{Reson}^{low}(v)\Big|_{\alpha=0} \right) = \mathcal{R}_m(v),\\
\{\mathcal{F}_m, \mathcal{L}_m \}(v) = 2 \mathrm{Im}(\sum_{k\geq 0} \overline{a}_{m,k,k}|v_k|^2 + \frac{1}{2}\sum_{j+k=2m}a_{j,m,k}v_j v_k) =-\sum_{n\geq 2m+1} |v_n|^2.
\end{cases}
\end{equation*}By $(\ref{Solution of linear system of all a _kln})$ with $\alpha=0$, we have $|a_{j,j+k-m,k}|\leq \frac{1}{2}$, for all $j,k \geq 0$. By the definition of $\mathcal{F}_m$, we have 
\begin{equation*}
\begin{cases}
\widehat{[X_{\mathcal{F}_m}(v)]}(n)= -2i\sum_{k-l+n=m}\overline{a}_{k,l,n} \overline{v}_k v_l -i\sum_{k-n+l=m}a_{k,n,l} v_k v_l\\
\widehat{[\mathrm{d} X_{\mathcal{F}_m}(v)h]}(n)= -2i\sum_{k-l+n=m}\overline{a}_{k,l,n} (\overline{v}_k h_l +v_l \overline{h}_k) -2i\sum_{k-n+l=m}a_{k,n,l} v_k h_l,
\end{cases}, \forall n \geq 0 .
\end{equation*}The last two estimates are obtained by Young's convolution inequality for $l^1 * l^2 \hookrightarrow l^2$.

\end{proof}

\subsubsection{The Birkhoff normal form}
\noindent Set $\chi^m_{\sigma}:=\exp(\epsilon \sigma X_{\mathcal{F}_m})$ the Hamiltonian flow of $\epsilon  {\mathcal{F}_m}$, i.e.,
\begin{equation*}
\frac{\mathrm{d}}{\mathrm{d}\sigma} \chi^m_{\sigma}(v)=\epsilon X_{\mathcal{F}_m}(\chi^m_{\sigma}(v)), \qquad \chi^m_{0}(v)=v.
\end{equation*}We perform the canonical transformation $\Psi_{m,\epsilon}:=\chi^m_1 = \exp(\epsilon X_{\mathcal{F}_m})$. We want to reduce the energy functional $\mathcal{H}^{m,\epsilon}$ to the following norm 
\begin{equation*}
\mathcal{H}^{m,\epsilon}(t) \circ\Psi_{m,\epsilon}= \mathcal{H}^m_0+ \varphi'(t) \mathcal{L}_m(v) + \epsilon \left(\mathcal{R}_m + \varphi'(t) (-\tilde{\mathcal{N}}_2+\frac{\mathcal{N}_2}{2}) \right)+ \mathcal{O}(\epsilon^2).
\end{equation*}Since $\mathcal{R}_m$ depends only on low frequency Fourier modes $v_1, \cdots, v_{3m}$, the high-frequency filtering $H^s$ norm of the solution of $\partial_t w(t) = X_{\mathcal{H}^{m,\epsilon}(t) \circ\Psi_{m,\epsilon}}(w(t))$ is handled by the Birkhoff normal form transformation. The estimate of $\|\mathbb{P}_{3m}(v)\|_{H^s}$ is given by $(\ref{uniform estimate of v})$. The next lemma will give the local existence of $\chi^m_{\sigma}$, for $|\sigma|\leq 1$ and estimate the difference between $v$ and $\Psi_{m,\epsilon}^{-1}(v)$.

\begin{lem}\label{estimates of hamiltonian flow of epsilon 1 F orbital stability m}
For every $s>\frac{1}{2}$ and $m \in \mathbb{N}$, there exist two constants $\gamma_{m,s}, C_{m,s}>0$ such that for all $v\in H^s_+$, if  $\epsilon  \|v\|_{H^s} \leq \gamma_{m,s}$, then $\chi^m_{\sigma}(v)$ is well defined on the interval $[-1,1]$ and the following estimates hold:
\begin{equation*}
\begin{split}
&\sup_{\sigma\in [-1,1]}\|\chi^m_{\sigma}(v)\|_{H^s} \leq 2\|v\|_{H^s},\\
&\sup_{\sigma\in [-1,1]}\|\chi^m_{\sigma}(v)-v\|_{H^s} \leq C_{m,s}  \epsilon  \|v\|^2_{H^s},\\
&\|\mathrm{d}\chi^m_{\sigma}(v)\|_{B(H^s)} \leq \exp(C_{m,s}\epsilon \|v\|_{H^s}|\sigma|),\qquad \forall \sigma \in [-1,1].
\end{split}
\end{equation*}
\end{lem}
\noindent The proof is based on a bootstrap argument, which is similar to $\mathbf{Lemma}$ $\mathbf{\ref{estimates of hamiltonian flow of epsilon 2 F}}$, given by $\mathbf{Lemma}$ $\mathbf{\ref{bootstrap}}$ with $m=2$. We shall perform the canonical transform below. Recall that $\Psi_{m,\epsilon}= \chi_1^m$.
\begin{lem}\label{Birkhoff normal form transformation for travelling wave alpha=0}
For all $s>\frac{1}{2}$, $m \in \mathbb{N} $ and $0<\epsilon<\epsilon^*_m$, there exists a smooth mapping $\mathcal{Y}_m : \mathbb{R}\times H^s_+ \longrightarrow H^s_+$ and a constant $C'_{m,s} >0$ such that for all $t \in \mathbb{R}$, we have
\begin{equation*}
X_{\mathcal{H}^{m,\epsilon}(t)\circ \Psi_{m,\epsilon}}=X_{\mathcal{H}^{m}_0}+\varphi'(t) X_{\mathcal{L}_m}+\epsilon \left( X_{\mathcal{R}_{m} }+\varphi'(t)(-X_{\tilde{\mathcal{N}}_2}+\frac{1}{2}X_{\mathcal{N}_2} ) \right)+\epsilon^{2} \mathcal{Y}_m(t),
\end{equation*}and $\sup_{t\in \mathbb{R}}\|\mathcal{Y}_m(t,v)\|_{H^s} \leq  C'_{m,s}\|v\|^2_{H^s}(1+\|v\|_{H^s})$, for all $v\in H^s_+$ such that $\epsilon \|v\|_{H^s}\leq \gamma_{m,s}$. Set $w(t):=\Psi_{m,\epsilon}^{-1}(v(t))$, then we have $\partial_t w(t)=X_{\mathcal{H}^{m,\epsilon}(t)\circ \Psi_{m,\epsilon}}(w(t))$ and 
\begin{equation}\label{estimates of w(t) epsilon alpha=0}
\Big|\frac{\mathrm{d}}{\mathrm{d}t} \|w(t)-\mathbb{P}_{3m}(w(t))\|_{H^s}^2 \Big| \leq C'_{m,s} \epsilon^{2} \|w(t)\|^3_{H^s}(1+\|w(t)\|_{H^s}), 
\end{equation}if $\epsilon \|w(t)\|_{H^s}\leq \gamma_{m,s}$.
\end{lem}

\begin{proof}
For every $t \in \mathbb{R}$, we expand the energy $\mathcal{H}^{m,\epsilon}(t)\circ \Psi_{m,\epsilon}=\mathcal{H}^{m, \epsilon}(t)\circ \chi^m_1$ with Taylor's formula at time $\sigma=1$ around $0$. Since $\chi^m_0 = \mathrm{Id}_{H^s_+}$, we have

\begin{equation*}
\begin{split}
&\left(\mathcal{H}^{m}_{0}+ \varphi'(t) \mathcal{L}_m \right)\circ \chi^m_1\\
 = &\mathcal{H}^{m}_0+\varphi'(t) \mathcal{L}_m+ \frac{\mathrm{d}}{\mathrm{d}\sigma} [(\mathcal{H}^{m}_0+\varphi'(t) \mathcal{L}_m)\circ  \chi^m_{\sigma}]\arrowvert_{\sigma=0}+\int_0^1(1-\sigma)\frac{\mathrm{d}^2}{\mathrm{d}\sigma^2}[(\mathcal{H}^{m}_0+\varphi'(t) \mathcal{L}_m)\circ \chi^m_{\sigma}]\mathrm{d}\sigma\\
 =& \mathcal{H}^{m}_0+\varphi'(t) \mathcal{L}_m + \epsilon \{\mathcal{F}_m, \mathcal{H}^{m}_0+\varphi'(t) \mathcal{L}_m\}+\epsilon^2 \int_0^1(1-\sigma)\{\mathcal{F}_m, \{\mathcal{F}_m, \mathcal{H}^{m}_0+\varphi'(t) \mathcal{L}_m \}\}\circ \chi^m_{\sigma}\mathrm{d}\sigma\\
\end{split}
\end{equation*}and 
\begin{equation*}
\begin{split}
\left(\mathcal{H}^{m}_1 +\frac{\varphi'(t)}{2} \mathcal{N}_2\right)\circ \chi^m_1 
=&\mathcal{H}^{m}_1+\frac{\varphi'(t)}{2} \mathcal{N}_2 +\int_0^1 \frac{\mathrm{d}}{\mathrm{d}\sigma}[(\mathcal{H}^{m}_1 +\frac{\varphi'(t)}{2} \mathcal{N}_2)\circ \chi^m_{\sigma}]\mathrm{d}\sigma\\
=& \mathcal{H}^{m}_1+\frac{\varphi'(t)}{2} \mathcal{N}_2 + \epsilon \int_0^1 \{\mathcal{F}_m, \mathcal{H}^{m}_1+\frac{\varphi'(t)}{2} \mathcal{N}_2 \}\circ \chi^m_{\sigma}\mathrm{d}\sigma.
\end{split}
\end{equation*}Since we have the homological system $\begin{cases}
\{\mathcal{F}_m, \mathcal{H}_0^{m}\} +\mathcal{H}^m_1 = \mathcal{R}_m \\
\{\mathcal{F}_m, \mathcal{L}_{m}\}+\tilde{\mathcal{N}}_2 = 0\\
\end{cases}$ in $\mathbf{Proposition}$ $\mathbf{\ref{homological equation alpha =0 prop}}$, we have

\begin{equation*}
\begin{split}
& \mathcal{H}^{m,\epsilon}(t)\circ \chi^m_1 \\
= & \mathcal{H}^{m}_0+\varphi'(t) \mathcal{L}_m + \epsilon\left(\{\mathcal{F}_m, \mathcal{H}^{m}_0\}+ \mathcal{H}^{m}_1+ \varphi'(t)(\{\mathcal{F}_m, \mathcal{L}_m\}+\frac{1}{2} \mathcal{N}_2)\right)\\
& \quad + \epsilon^2 \left[\int_0^1 \{\mathcal{F}_m, (1-\sigma)\{\mathcal{F}_m, \mathcal{H}^{m}_0+\varphi'(t) \mathcal{L}_m \} + \mathcal{H}^{m}_1+\frac{\varphi'(t)}{2} \mathcal{N}_2  \}\circ \chi^m_{\sigma}\mathrm{d}\sigma + \frac{\mathcal{N}_4 \circ \chi_1^m}{4} \right]\\
=& \mathcal{H}^{m}_0+\varphi'(t) \mathcal{L}_m + \epsilon\left( \mathcal{R}_m + \varphi'(t)(-\tilde{\mathcal{N}}_2+\frac{1}{2} \mathcal{N}_2)\right) + \epsilon^2\left[\int_0^1 \mathcal{G}_m(t,\sigma) \circ\chi^m_{\sigma}\mathrm{d}\sigma + \frac{\mathcal{N}_4 \circ \chi_1^m}{4} \right],\\
\end{split}
\end{equation*}where $\mathcal{G}_m(t,\sigma) = \{\mathcal{F}_m, (1-\sigma) \mathcal{R}_m + \sigma  \mathcal{H}^{m}_1+\varphi'(t)((\sigma-1)\tilde{\mathcal{N}}_2 +\frac{1}{2} \mathcal{N}_2)  \}$. We set
\begin{equation*}
\mathcal{Y}_m(t,v):=\int_0^1 X_{\mathcal{G}_m (t,\sigma)\circ \chi^m_{\sigma}}(v)\mathrm{d}\sigma+\frac{1}{4} X_{\mathcal{N}_4 \circ \chi^m_1}(v), 
\end{equation*}then we get $X_{\mathcal{H}^{m,\epsilon}(t)\circ \chi^m_1}=X_{\mathcal{H}^{m}_0}+\varphi'(t) X_{\mathcal{L}_m}+\epsilon \left( X_{\mathcal{R}_{m} }+\varphi'(t)(-X_{\tilde{\mathcal{N}}_2}+\frac{1}{2}X_{\mathcal{N}_2} ) \right)+\epsilon^{2} \mathcal{Y}_m(t)$.

\noindent Since $\mathcal{F}_m$, $\mathcal{H}^1_m$ and $\mathcal{R}_m$ are homogeneous series of order $3$ with uniformly bounded coefficients, $\mathcal{N}_2$ and $\tilde{\mathcal{N}}_2$ are homogeneous series of order $2$ with uniformly bounded coefficients, we have
\begin{equation*}
\begin{cases}
\|X_{\{\mathcal{F}_m, \mathcal{H}^1_m\}}(v)\|_{H^s}+\|X_{\{\mathcal{F}_m, \mathcal{R}_m\}}(v)\|_{H^s} \lesssim_s  \|v\|_{H^s}^3,\\
\|X_{\{\mathcal{F}_m, \mathcal{N}_2\}}(v)\|_{H^s}+\|X_{\{\mathcal{F}_m, \tilde{\mathcal{N}}_2\}}(v)\|_{H^s} \lesssim_s  \|v\|_{H^s}^2,\\
\end{cases}
\end{equation*}because for $\mathcal{J}_m(v)=\sum_{j-l+k=m}\mathrm{Re}(b_{j,l,k}v_j \overline{v}_l v_k)$ with $\sup_{j-l+k=m}|b_{j,l,k}| <+\infty$, we have
\begin{equation*}
\begin{split}
&\{\mathcal{F}_m, \mathcal{J}_m\}(v)\\
=& 4 \mathrm{Im}\sum_{n\geq 0} \partial_{\overline{v}_n} \mathcal{F}_m(v)\partial_{v_n} \mathcal{J}_m(v)\\
=& \sum_{k_1+k_2=l_1+l_2} \mathrm{Im}(4\overline{a}_{k_1,l_1,m+l_1-k_1}b_{l_2,k_2,m+k_2-l_2}+a_{l_1,l_1+l_2-m,l_2}\overline{b}_{k_1,k_1+k_2-m, k_2 })\overline{v}_{k_1}\overline{v}_{k_2}  v_{l_1}v_{l_2}\\
& +\sum_{k_1+k_2+l_1-l_2=2m} 2\mathrm{Im}( a_{k_1,k_1+l_1-m,l_1}b_{k_2,l_2,m+l_2-k_2} - a_{k_1,l_2,m+l_2-k_2} b_{k_2,l_1+k_2-m, l_1 })v_{k_1}v_{k_2}  v_{l_1}\overline{v}_{l_2}\\
\end{split}
\end{equation*}and $\{\mathcal{F}_m, \mathcal{N}_2\}(v)=-2 \mathrm{Im}(\sum_{j-l+k=m}a_{j,l,k} v_j \overline{v}_l v_k)$. Recall that $\sup_{0<\epsilon < \epsilon^*_m }\sup_{t\in \mathbb{R}}|\varphi'(t)|<+\infty$ and $X_{\mathcal{N}_4}(v)=-4i \Pi(|v|^2v)$, then we have 
\begin{equation*}
\sup_{0\leq \sigma \leq 1}\sup_{t\in \mathbb{R}} \|X_{\mathcal{G}_m(t,\sigma)}(v)\|_{H^s}+\|X_{\mathcal{N}_4}(v)\|_{H^s} \lesssim_{m,s } \|v\|^2_{H^s}(1+\|v\|_{H^s}) .
\end{equation*}By using $\mathbf{Lemma}$ $\mathbf{\ref{estimates of hamiltonian flow of epsilon 1 F orbital stability m}}$ and $(\ref{composition symplectic gradient})$, for all $v \in H^s_+$ such that $\epsilon \|v\|_{H^s} \leq \gamma_{m,s}$, we have
\begin{equation*}
\begin{split}
\sup_{0\leq \sigma \leq 1}\sup_{t\in \mathbb{R}}\|X_{\mathcal{G}_m(t,\sigma)\circ \chi^m_{\sigma}}(v)\|_{H^s} \leq & \sup_{0\leq \sigma \leq 1}\sup_{t\in \mathbb{R}}\|\mathrm{d}\chi^m_{-\sigma}(\chi^m_{\sigma}(v))\|_{B(H^s)}\|X_{\mathcal{G}_m(t,\sigma)}(\chi^m_{\sigma}(v))\|_{H^s} \\
\lesssim_{m,s} &  \sup_{0\leq \sigma \leq 1} e^{C_{m,s}\epsilon \|\chi^m_{\sigma}(v)\|_{H^s}}\|\chi^m_{\sigma}(v)\|^2_{H^s}(1+\|\chi^m_{\sigma}(v)\|_{H^s} )\\
\lesssim_{m,s} & \|v\|_{H^s}^2(1+\|v\|_{H^s} )
\end{split}
\end{equation*}and $\sup_{t\in \mathbb{R}}\|X_{\mathcal{N}_4\circ \chi^m_1}(v)\|_{H^s} \lesssim_{m,s}  \|v\|^2_{H^s}(1+\|v\|_{H^s})$. Then we obtain the estimate of $\mathcal{Y}_m$.\\

\noindent Since $w(t)=\chi_{-1}^m (v(t))$, we have the following infinite dimensional Hamiltonian system on the Fourier modes:
\begin{equation*}
i\partial_t w_n(t) =(1+ n^2-m^2-\epsilon \varphi'(t))w_n(t) + i\epsilon^{2}\widehat{\mathcal{Y}_m(t,w(t))}(n), \quad \forall n \geq 3m+1.
\end{equation*}because for all $n \geq 3m+1$, we have
\begin{equation*}
\begin{cases}\widehat{H_{e^{2imx}}(w(t))}(n)=\widehat{X_{\mathcal{R}_m}(w(t))}(n)=\widehat{X_{\mathcal{L}_m}(w(t))}(n)=0,\\
\widehat{X_{\mathcal{N}_2}(w(t))}(n)=\widehat{X_{\tilde{\mathcal{N}}_2}(w(t))}(n)=-2i w_n(t).
 \end{cases} 
\end{equation*}Consequently, if $\epsilon \|w(t)\|_{H^s}\leq \gamma_{m,s}$, then we have
\begin{equation*}
\begin{split}
\Big|\partial_t \|w(t)-\mathbb{P}_{3m}(w(t))\|_{H^s}^2 \Big|\leq & 2  \epsilon^{2} \sum_{n\geq 3m+1} (1+n^2)^s|\widehat{\mathcal{Y}_m(t,w(t))}(n)||w_n(t)| \\
 \leq & 2 \epsilon^{2} \|\mathcal{Y}_m(t,w(t))\|_{H^s}\|w(t)\|_{H^s}\\
 \lesssim_{m,s} &  \epsilon^{2} \|w(t)\|^3_{H^s}(1+\|w(t)\|_{H^s}).
\end{split}
\end{equation*}

\end{proof}

\subsubsection{End of the proof of $\mathbf{Theorem}$ $\mathbf{\ref{effective dynamic of equation NLSF with u0=exp(ix)+epsilon, t leq epsilon^(-2)}}$}
The proof is completed by a bootstrap argument and estimate $(\ref{estimates of w(t) epsilon alpha=0})$, obtained by the Birkhoff normal form transformation. It suffices to prove the case $u(0) \in C^{\infty}_+$ by the same density argument in the proof of $\mathbf{Proposition}$ $\mathbf{\ref{epsilon ^ (alpha/2 -1) time estimate for NLS Szego epsilon alpha}}$.
\begin{proof}For all $m \in  \mathbb{N} $ and $s\geq 1$, we recall that $u(t,x)=e^{i\arg u_m(t)}(e^{imx}+\epsilon  v(t,x))$, $\partial_t v(t)= X_{\mathcal{H}^{m,\epsilon}(t)}(v(t))$ and $w(t)= \chi^m_{-1}(v(t))$. In $\mathbf{Proposition}$ $\mathbf{\ref{proposition to define theta varphi alpha<<1}}$, we have shown that there exists $A_{m,s} \geq 1$ such that $\sup_{0 <\epsilon < 1}\|v(0)\|_{H^s} \leq A_{m,s}$. By using $(\ref{uniform estimate of v})$, we have 
\begin{equation*}
\sup_{0<\epsilon<1}\sup_{t\in \mathbb{R}}\|\mathbb{P}_{3m}(v(t))\|_{H^s} \leq (1+9m^2)^{\frac{s}{2}}I_m.
\end{equation*}Set $K_{m,s}:= \max(6A_{m,s}, 6(1+9m^2)^{\frac{s}{2}}I_m)$, $\epsilon_{m,s}:=\min(\epsilon^*_m,\frac{\gamma_{m,s}}{3K_{m,s}}, \frac{1}{48C_{m,s}K_{m,s}})$ and
\begin{equation*}
T_{m,s}:=\sup\{t\geq 0 :  \sup_{0\leq \tau  \leq t}\|v(\tau)\|_{H^s} \leq 2K_{m,s}\}.
\end{equation*}We choose $\epsilon\in (0,\epsilon_{m,s})$. Since $\epsilon = \epsilon \|v(0)\|_{H^s} \leq \epsilon A_{m,s} \leq \gamma_{m,s}$, $\mathbf{Lemma}$ $\mathbf{\ref{estimates of hamiltonian flow of epsilon 1 F orbital stability m}}$ yields that 
\begin{equation*}
\|v(0)-w(0)\|_{H^s}=\|v(0)-\chi^m_{-1}(v(0))\|_{H^s}\leq  \epsilon   A^2_{m,s}C_{m,s} \leq A_{m,s} . 
\end{equation*}So we have $\|w(0)\|_{H^s} \leq \frac{K_{m,s}}{3}$. For all $t\in [0,T_{m,s}]$, we have $\epsilon  \|v(t)\|_{H^s} \leq {\gamma_{m,s}}$. Hence $\mathbf{Lemma}$ $\mathbf{\ref{estimates of hamiltonian flow of epsilon 1 F orbital stability m}}$ gives the following estimate
\begin{equation*}
\begin{split}
\|w(t)\|_{H^s} \leq & \|v(t)\|_{H^s} + \|\chi^m_{-1}(v(t))-v(t)\|_{H^s}\\
\leq & \|v(t)\|_{H^s} + C_{m,s} \epsilon \|v(t)\|_{H^s}^2\\
\leq & 2K_{m,s} + 4C_{m,s} K_{m,s}^2 \epsilon  \\
\leq & 3K_{m,s}.
\end{split}
\end{equation*}So we have $\epsilon \sup_{0\leq t \leq T_{m,s}} \|w(t)\|_{H^s}  \leq \gamma_{m,s}$, which implies that
\begin{equation*}
\Big|\frac{\mathrm{d}}{\mathrm{d}t} \|w(t)-\mathbb{P}_{3m}(w(t))\|_{H^s}^2 \Big| \leq C'_{m,s} \epsilon^{2} \|w(t)\|^3_{H^s}(1+\|w(t)\|_{H^s}), 
\end{equation*}in $\mathbf{Lemma}$ $\mathbf{\ref{Birkhoff normal form transformation for travelling wave alpha=0}}$. Set $d_{m,s}:=\frac{1}{486 K_{m,s}^2 C'_{m,s}}$. We can obtain the following estimate:
\begin{equation*}
\begin{split}
&\|w(t)-\mathbb{P}_{3m}(w(t))\|_{H^s}^2\\
 \leq &\|w(0)\|_{H^s}^2 + C'_{m,s} |t|   \epsilon^{2} \sup_{0\leq \tau\leq T_{m,s}}\|w(\tau)\|^3_{H^s}(1+\sup_{0\leq \tau\leq T_{m,s}}\|w(\tau)\|_{H^s})\\
\leq & \frac{K_{m,s}^2}{9} + 162 C'_{m,s} K_{m,s}^4 |t| \epsilon^{2}\\
\leq & \frac{4K_{m,s}^2}{9},
\end{split}
\end{equation*}for all $0 \leq t \leq \min(T_{m,s},  \frac{d_{m,s}}{\epsilon^{2 }})$. We use $\mathbf{Lemma}$ $\mathbf{\ref{estimates of hamiltonian flow of epsilon 1 F orbital stability m}}$ again to obtain that
\begin{equation*}
\begin{split}
\|v(t)\|_{H^s} \leq & 2\|w(t)-v(t)\|_{H^s}+\|w(t)-\mathbb{P}_{3m}(w(t))\|_{H^s} + \|\mathbb{P}_{3m}(v(t))\|_{H^s} \\
 \leq & 2 C_{m,s} \epsilon  \|v(t)\|_{H^s}^2+ \frac{2K_{m,s}}{3}+\frac{K_{m,s}}{6}\\
 \leq & 8 C_{m,s} K_{m,s}^2 \epsilon+\frac{5K_{m,s}}{6}\\
 \leq & K_{m,s},
\end{split}
\end{equation*}for all $t \in [0, \frac{d_{m,s}}{\epsilon^{2 }}]$. In the case $t<0$, we use the same procedure and we replace $t$ by $-t$. Consequently, we have
\begin{equation*}
\sup_{| t |\leq \frac{d_{m,s}}{\epsilon^{2 }}}\|u(t)-e^{i(\cdot+\arg u_m(t))}\|_{H^s}= \epsilon  \sup_{| t| \leq \frac{d_{m,s}}{\epsilon^{2 }}}\|v(t)\|_{H^s} \leq K_{m,s}   \epsilon .
\end{equation*}
\end{proof}

\section{Comparison to NLS equation}\label{comparison NLS NLSSZEGO}
Although we have some similar results for the NLS equation,
there are still some differences between the NLS equation and the NLS-Szeg\H{o} equation. We denote by $u=u(t,x) =\sum_{n\geq 0} u_n(t) e^{inx}$ the solution of the NLS equation 
\begin{equation}\label{NLS comparation}
i\partial_t u + \partial_x^2 u = |u|^2 u.
\end{equation}In Fourier modes, we have $i\partial_t u_n = n^2 u_n + \sum_{k_1-k_2+k_3=n} u_{k_1} \overline{u}_{k_2} u_{k_3}$. Fix $m \in  \mathbb{Z}$, for every $n \in \mathbb{Z}$, we define $v_n:=u_{n+m}e^{i(m^2+2nm)t}$. Then $\|v(t)\|_{L^2}=\|u(t)\|_{L^2}$ and we have 
\begin{equation}\label{v equation for comparation NLS NLSSzego}
i\partial_t v_n = n^2 v_n + \sum_{k_1-k_2+k_3=n} v_{k_1} \overline{v}_{k_2} v_{k_3}.
\end{equation}If $u$ is localized in the $m$-th Fourier mode, then $v$ is localized on the zero mode. Thus the orbital stability of the traveling wave $\mathbf{e}_m$ can be reduced to the case $m=0$. In Faou--Gauckler--Lubich $[\mathbf{\ref{faou--Gauckler--lubich Sobolev Stability of Plane Wave}}]$, $H^s$-orbital stability of plane wave solutions is established by limiting the mass of the initial data $\|u_0\|_{H^s}$ to a certain full measure subset of $(0,+\infty)$ for the defocusing cubic Schr\"odinger equation with the time interval $[-\epsilon^{-N},\epsilon^{-N}]$, for all $N \geq 1$ and $s \gg 1$. However, this above transformation $u\mapsto v$ does not preserve the $L^2$ norm for the NLS-Szeg\H{o} equation and formula $(\ref{v equation for comparation NLS NLSSzego})$ fails too. \\

\noindent On the other hand, the approach that we use to prove $\mathbf{Theorem}$ $\mathbf{\ref{effective dynamic of equation NLSF with u0=exp(ix)+epsilon, t leq epsilon^(-2)}}$ does not work for $(\ref{NLS comparation})$. In fact, the negative high frequency Fourier modes in the term $\mathcal{H}^m_1(v)=\mathrm{Re}\int_{\mathbb{S}^1}e^{-imx} |v|^2 v$ can not be cancelled by the homological equation. The energy functional of the equation of $v$ can not be reduced as $ \mathcal{H}_0^{m} + \mathcal{O}(\epsilon^2)$ by using the same method in this paper. The Szeg\H{o} filtering to $(\ref{NLS comparation})$ makes it possible to cancel all the high frequency resonances in the term $\mathcal{R}_m=\{\mathcal{F}_m, \mathcal{H}_0^m\} + \mathcal{H}^m_1$. Then we use a bootstrap argument to deal with the equation $\partial_t w(t) = X_{ \mathcal{H}^{m,\epsilon}(t) \circ \chi_1^m}(w(t))$ after the Birkhoff normal form transformation. 

\section{Appendix}\label{appendix}
We give the details of the homological equation in $\mathbf{Proposition}$ $\mathbf{\ref{homological equation alpha =0 prop}}$ and prove $(\ref{Solution of linear system of all a _kln})$ and $\mathbf{Proposition}$ $\mathbf{\ref{homological equation for all alpha}}$ in the case $\alpha=0$.\\

\noindent For all $v\in \bigcup_{n\in \mathbb{N}} \mathbb{P}_n (C^{\infty}_+)$, $\mathcal{H}_0^{m}(v) = \frac{1}{2}\|\partial_x v\|_{L^2}^2 + \frac{1-m^2 }{2}\|v\|_{L^2}^2 + \frac{1}{2}\int_{\mathbb{S}^1}\mathrm{Re}(e^{2imx} \overline{v}^2) \mathrm{d}x$ and ${\mathcal{F}_m}(v)=\sum_{\begin{smallmatrix}
 j-l+k=m\\ j,k,l \in \mathbb{N}
\end{smallmatrix}}\mathrm{Re}(a_{j,l,k}v_j \overline{v}_l v_k)$. With the convention $v_n = 0$, for all $n<0$, we have
\begin{equation*}
\begin{cases}
\partial_{v_n}\mathcal{H}^{m}_0(v)=\overline{\partial_{\bar{v}_n}\mathcal{H}^{m}_0(v)}=\tfrac{1+n^2-m^2 }{2}\overline{v}_n + \tfrac{1}{2}v_{2m-n},\\
\partial_{\bar{v}_n}\mathcal{F}_m(v)=\overline{\partial_{v_n}\mathcal{F}_m(v)}=\sum_{k-l+n=m}\bar{a}_{k,l,n}\overline{v}_k v_l  + \tfrac{1}{2}\sum_{k-n+l=m}{a}_{k,n,l} v_k v_l. \\
\end{cases}
\end{equation*}Combing $(\ref{Poisson bracket})$, we have the Poisson bracket of $\mathcal{F}_m$ and $\mathcal{H}^{m}_0 $,
\begin{equation*}
\begin{split}
&\{\mathcal{F}_m,\mathcal{H}^{m}_0\}(v)\\ =&  -2i \sum_{n\geq 0} \left( \partial_{\overline{v}_n} \mathcal{F}_m  \partial_{v_n} \mathcal{H}^{m }_0 - \partial_{\overline{v}_n }\mathcal{H}^{m }_0 \partial_{v_n}\mathcal{F}_m\right)(v)\\
=&\mathrm{Im}(\sum_{ \begin{smallmatrix} 
k-l+n=m  \\
k,l,n \in \mathbb{N}
\end{smallmatrix}}2(1+n^2-m^2 )  \overline{a}_{k,l,n}\overline{v}_k v_l \overline{v}_n+\sum_{ \begin{smallmatrix} 
k-n+l=m  \\
k,l,n \in \mathbb{N}
\end{smallmatrix}}(1+ n^2-m^2 ) a_{k,n,l}v_k \overline{v}_n v_l ) \\
&+\mathrm{Im}(\sum_{ \begin{smallmatrix} 
k-l+n=m  \\
k,l,n \in \mathbb{N}, n\leq 2m
\end{smallmatrix}}2  \overline{a}_{k,l,n}v_l \overline{v}_k v_{2m-n} + \sum_{ \begin{smallmatrix} 
k-n+l=m  \\
k,l,n \in \mathbb{N}, n \leq 2m
\end{smallmatrix}}a_{k,n,l}v_k v_l v_{2m-n})\\
=& \mathrm{Im}(\mathcal{A}_1+\mathcal{A}_2+\mathcal{A}_3+\mathcal{A}_4).
\end{split}
\end{equation*}The term $\mathcal{A}_4=  \sum_{ \begin{smallmatrix} 
k-n+l=m  \\
k,l,n \in \mathbb{N}, n \leq 2m
\end{smallmatrix}}a_{k,n,l}v_k v_l v_{2m-n}$ has only finite term depending on low frequency Fourier modes $v_1, \cdots, v_{3m}$. We divide $\mathcal{A}_1,\mathcal{A}_2,\mathcal{A}_3$ into two parts. The first part consists of only low frequency resonances, the second part consists of high frequency resonances.

\begin{equation*}
\mathcal{A}_1=  -\sum_{ \begin{smallmatrix} 
k-l+n=m  \\
k,l,n \in \mathbb{N}
\end{smallmatrix}}2(1+ n^2-m^2 )a_{k,l,n}v_k \overline{v}_l v_n = \mathcal{A}_1^{low} + \mathcal{A}_1^{\geq 2m+1},
\end{equation*}where $\mathcal{A}_1^{low}$ consists of all the resonances $v_j \overline{v}_l v_k$ such that $j,k \leq 2m$,
\begin{equation*}
\begin{split}
\mathcal{A}_1^{low} 
 = & -(\sum_{j=0}^{\lfloor \frac{m-1}{2}\rfloor} \sum_{k=m-j}^{2m} + \sum_{j=\lfloor \frac{m+1}{2}\rfloor}^{2m-1} \sum_{k=j+1}^{2m})2(2+ j^2+k^2-2m^2 )a_{j,j+k-m,k}v_j \overline{v}_{j+k-m} v_k\\
& -\sum_{j = 0}^{m-1}\sum_{k=j+m+1}^{2m}2(2+ j^2+(k+m-j)^2-2m^2  )a_{j,k,k+m-j}v_j \overline{v}_{k} v_{k+m-j}\\
& -\sum_{k = \lfloor \frac{m+1}{2} \rfloor}^{2m}2(1+k^2-m^2 )a_{k,2k-m,k}v_k^2 \overline{v}_{2k-m},\\
 \end{split}
\end{equation*}and $\mathcal{A}_1^{\geq 2m+1}$ contain other resonances $v_j \overline{v}_l v_k$ such that at least one of $j,k$ is strictly larger than $2m$.
\begin{equation*}
\begin{split}
\mathcal{A}_1^{\geq 2m+1}= & - \sum_{j = 0}^{m} \sum_{k\geq 2m+1} 2(2+  j^2+(k+m-j)^2-2m^2  )a_{j,k,k+m-j}v_j \overline{v}_{k} v_{k+m-j}\\
& -  \sum_{j = m+1}^{2m}\sum_{k\geq 2m+1}2(2+j^2+k^2-2m^2)a_{j,j+k-m,k}v_j \overline{v}_{j+k-m} v_k\\
& - \sum_{k\geq 2m+1}\sum_{\begin{smallmatrix}
 j \geq k+1 
\end{smallmatrix}} 2(2+ k^2+j^2-2m^2)a_{k,k+j-m,j}v_k \overline{v}_{k+j-m} v_j\\
& - \sum_{k\geq 2m+1}2(1+k^2-m^2)a_{k,2k-m,k}v_k^2 \overline{v}_{2k-m}.
\end{split}
\end{equation*}Then, we calculate $\mathcal{A}_2=   \sum_{ \begin{smallmatrix} 
k-n+l=m  \\
k,l,n \in \mathbb{N}
\end{smallmatrix}}(1+ n^2-m^2 )a_{k,n,l}v_k \overline{v}_n v_l =\mathcal{A}_2^{low}+\mathcal{A}_2^{\geq 2m+1}$. $\mathcal{A}_2^{low}$ consists of all the resonances $v_k \overline{v}_l v_n$ such that $k,n \leq 2m$ or $l \leq 2m$.
\begin{equation*}
\begin{split}
\mathcal{A}_2^{low} =&\sum_{ \begin{smallmatrix} 
j-l+k=m  \\
j,l,k \in \mathbb{N}, l\leq 2m
\end{smallmatrix}}(1+ l^2-m^2  )a_{j,l,k}v_j \overline{v}_l v_k\\
&  \quad + \sum_{ \begin{smallmatrix} 
j-l+k=m  \\
m+1 \leq j \leq 2m-1\\
j+1 \leq k\leq 2m, l \geq 2m+1\\
\end{smallmatrix}}2(1+ l^2-m^2 )a_{j,l,k}v_j \overline{v}_l v_k\\
& \quad + \sum_{k=\lfloor \frac{3m}{2} \rfloor+1}^{2m} (1+(2k-m)^2-m^2 )a_{k,2k-m,k} v_k^2 \overline{v}_{2k-m};\\
\end{split}
\end{equation*}$\mathcal{A}_2^{\geq 2m+1}$ plays the same role as $\mathcal{A}_1^{\geq 2m+1}$.

\begin{equation*}
\begin{split}
\mathcal{A}_2^{\geq 2m+1} = &   \sum_{j=0}^m \sum_{k\geq 2m+1} 2(1+ k^2-m^2  )a_{j,k,m+k-j}v_j \overline{v}_k v_{m+k-j}\\
& \quad+\sum_{j=m+1}^{2m} \sum_{k\geq 2m+1}2(1+ (j+k-m)^2-m^2  )a_{j,j+k-m,k}v_j \overline{v}_{j+k-m} v_k\\
& \quad+ \sum_{k\geq 2m+1} \sum_{j\geq k+1}2(1+ (k+j-m)^2-m^2  )a_{k,k+j-m,n}v_k \overline{v}_{k+j-m} v_j\\
& \quad+ \sum_{k\geq 2m+1}(1+ (2k-m)^2-m^2  )a_{k,2k-m,k} v_k^2 \overline{v}_{2k-m}.
\end{split}
\end{equation*}At last, $ \mathcal{A}_3 =  \sum_{ \begin{smallmatrix} 
k-l+n=m  \\
k,l,n \in \mathbb{N}, n \leq 2m
\end{smallmatrix}}2\overline{ a}_{k,l,n}v_l \overline{v}_k v_{2m-n}
= \sum_{ \begin{smallmatrix} 
j-l+k=m  \\
j, k,l,n \in \mathbb{N}, j \leq 2m
\end{smallmatrix}} 2 \overline{a}_{l,k,2m-j} v_k \overline{v}_l v_j$. Using the same idea, we have $ \mathcal{A}_3= \mathcal{A}_3^{low}+ \mathcal{A}_3^{\geq 2m+1}$, with
\begin{equation*}
\begin{split}
\mathcal{A}_3^{low} = \sum_{j= 0}^m \sum_{k=0}^{2m} 2 \overline{a}_{2m-j,m+k-j,k}v_j \overline{v}_k v_{m+k-j} + \sum_{j= m+1}^{2m} \sum_{k=0}^{2m} 2 \overline{a}_{2m-j,k,k+j-m}v_j \overline{v}_{k+j-m} v_{k};
\end{split}
\end{equation*}and

\begin{equation*}
\begin{split}
\mathcal{A}_3^{\geq 2m+1}=  \sum_{j= 0}^m \sum_{k \geq  2m+1} 2 \overline{a}_{2m-j,m+k-j,k}v_j \overline{v}_k v_{m+k-j} + \sum_{j= m+1}^{2m} \sum_{k \geq  2m+1} 2 \overline{a}_{2m-j,k,k+j-m}v_j \overline{v}_{k+j-m} v_{k}. \\
\end{split}
\end{equation*}Recall that $\mathcal{H}_{1}^m(v)=\mathrm{Re}(\int_{\mathbb{S}^1}e^{-imx}|v|^2v) = \sum_{k-l+n=m}\mathrm{Re}( v_k \overline{v}_l v_n)$. A similar calculus as in the case of $\mathcal{A}_1$ shows that $\mathcal{H}_{1 }^m( v) = \mathrm{Re}(B^{low} + B^{\geq 2m+1})$ with
\begin{equation*}
\begin{split}
& B^{low}\\
=&(\sum_{j=0}^{\lfloor \frac{m-1}{2}\rfloor} \sum_{k=m-j}^{2m} + \sum_{j=\lfloor \frac{m+1}{2}\rfloor}^{2m-1} \sum_{k=j+1}^{2m})2v_j \overline{v}_{j+k-m} v_k  +\sum_{k = \lfloor \frac{m+1}{2} \rfloor}^{2m} v_k^2 \overline{v}_{2k-m}+\sum_{j = 0}^{m-1}\sum_{k=j+m+1}^{2m}2 v_j \overline{v}_{k} v_{k+m-j},
 \end{split}
\end{equation*}
\begin{equation*}
\begin{split}
B^{\geq 2m+1}=  \sum_{k\geq 2m+1}\left(\sum_{j = 0}^{m}2 v_j \overline{v}_{k} v_{k+m-j}+\sum_{j = m+1}^{2m} 2v_j \overline{v}_{j+k-m} v_k + v_k^2 \overline{v}_{2k-m}+\sum_{\begin{smallmatrix}
 j \geq k+1 
\end{smallmatrix}} 2 v_k \overline{v}_{k+j-m} v_j \right).
\end{split}
\end{equation*}At last we define $\mathrm{Reson}^{low}(v)\Big|_{\alpha=0}=-i (\mathcal{A}_1^{low}+\mathcal{A}_2^{low}+\mathcal{A}_3^{low}+\mathcal{A}_4)+B^{low}$ and 
\begin{equation*}
\mathrm{Reson}^{\geq 2m+1}(v)\Big|_{\alpha=0}=-i (\mathcal{A}_1^{\geq 2m+1}+\mathcal{A}_2^{\geq 2m+1}+\mathcal{A}_3^{\geq 2m+1})+B^{\geq 2m+1}.
\end{equation*}Then we have $\{\mathcal{F}_m, \mathcal{H}^m_0\}(v)+\mathcal{H}_{1 }^m( v) = \mathrm{Re}(\mathrm{Reson}^{low}(v)\Big|_{\alpha=0}+ \mathrm{Reson}^{\geq 2m+1}(v)\Big|_{\alpha=0})$. Since $\mathrm{Reson}^{low}(v)\Big|_{\alpha=0}$ contains only finite terms and depends only on $v_1, \cdots, v_{3m}$, so is $\mathcal{R}_m=\mathrm{Re}(\mathrm{Reson}^{low}(v)\Big|_{\alpha=0})$. For high frequency resonances, we compute 

\begin{equation*}
\begin{split}
&  \mathrm{Reson}_{\geq 2m+1}(v)\Big|_{\alpha=0}\\
=&\sum_{k\geq 2m+1}\sum_{j=0}^{m-1} 2\left(-i \overline{a}_{2m-j,m+k-j,k}+(1+2(m-j)(k-j) )i a_{j,k,k+m-j}+1 \right)v_j \overline{v}_k v_{k+m-j}\\
+&\sum_{k\geq 2m+1}\sum_{j=m+1}^{2m} 2((1-2(k-m)(j-m) ) i a_{j,j+k-m,k}-i \overline{a}_{2m-j,k,k+j-m}+1)v_{j} \overline{v}_{j+k-m} v_k\\
+&\sum_{k\geq 2m+1}\left[2( i a_{m,k,k}-i \overline{a}_{m,k,k}+1)v_{m} |v_k|^2+((1-2(k-m)^2 )i a_{k,2k-m,k}+1)v_k^2\overline{v}_{2k-m}\right]\\
+&\sum_{k\geq 2m+1}\sum_{j \geq k+1}2 ((1-2(j-m)(k-m) )i a_{k,j+k-m,j}+1) v_{k} \overline{v}_{j+k-m} v_j.
\end{split}
\end{equation*}After replacing $j$ by $2m-j$ in the sum $m+1 \leq j \leq 2m$, we have the equivalence between $\mathrm{Reson}_{\geq 2m+1}\Big|_{\alpha=0}= 0$ and $(\ref{Solution of linear system of all a _kln})$ in the case $\alpha=0$.





\begin{thebibliography}{29}
\bibitem{Bambusi birkhoff normal form for nonlinear pde}Bambusi D., \emph{Birkhoff normal form for some nonlinear PDEs}, Comm. Math. Physics 234 (2003), 253–283. \label{Bambusi birkhoff normal form for nonlinear pde}
\bibitem{Bo}Bourgain, J., \emph{Fourier transform restriction phenomena for certain lattice subsets and applications to nonlinear evolution equations}, Geometric and Functional Analysis, March 1993, Volume 3, Issue 2, pp 107–156\label{bourgain strichartz inequality}

\bibitem{Boyer, Franck; Fabrie, Pierre}Boyer, F.; Fabrie, P., \emph{Mathematical Tools for the Study of the Incompressible Navier-Stokes Equations and Related Models}. Applied Mathematical Sciences 183. New York: Springer. pp. 102–106. (2013). \label{Boyer}


\bibitem{Brezis, Haim; Gallouet, Thierry}Brezis, H.; Gallou\"et, T., \emph{Nonlinear Schr\"odinger evolution equations}. Nonlinear Anal. 4, 677-681 \label{Brezis Gallouet inequality article}


\bibitem{i group 1}Colliander, J., Keel, M., Staffilani, G.,Takaoka, H., Tao, T., \emph{Transfer of energy to high frequencies in the cubic defocusing nonlinear Schrodinger equation}, Invent. Math.181 (2010), 39–113.\label{Colliander J., Keel M., Staffilani G.,Takaoka H., Tao Transfer of energy }

\bibitem{eliasson-kuksin kam for nls}Eliasson H., Kuksin S.,\emph{KAM for the nonlinear Schrödinger equation}, Ann. Math. 172(1), 371–435 (2010)\label{eliasson-kuksin kam for nls}
\bibitem{Hamiltonian Methods in the Theory of Solitons Faddeev-Takhtajan}Faddeev, L.D., Takhtajan, L.A.,\emph{Hamiltonian Methods in the Theory of
Solitons}, Springer series in Soviet Mathematics, Springer, Berlin, 1987.\label{Faddeev-Takhtajan}

\bibitem{faou--Gauckler--lubich}Faou E., Gauckler L., Lubich C.,\emph{Sobolev Stability of Plane Wave Solutions to the Cubic Nonlinear Schrödinger Equation on a Torus}, Comm. Partial Differential
Equations 38 (2013), no. 7, 1123–1140\label{faou--Gauckler--lubich Sobolev Stability of Plane Wave}


\bibitem{Gallay--Haragus stability of small periodic waves}Gallay, T., Haragus, M., \emph{Stability of small periodic waves for the nonlinear Schr\"odinger equation}, J. Diff. Eqns $\mathbf{234}$, 544-581 (2007)\label{Gallay--Haragus stability of small periodic waves}

\bibitem{Gallay--Haragus ORBITAL stability of periodic waves}Gallay, T., Haragus, M., \emph{Orbital stability of periodic waves for the nonlinear Schr\"odinger equation}, J. Diff. Eqns $\mathbf{19}$, 825-865 (2007)\label{Gallay--Haragus ORBITAL stability of periodic waves}

\bibitem{Gerard}Gérard, P. \emph{On the conservation laws of the
defocusing cubic NLS equation}, 2015, http://www.math.u-psud.fr/~pgerard/integrales-NLScubique.pdf\label{Gerard defocusing NLS integrals}
\bibitem{GerardGrellier1}Gérard, P.; Grellier, S. \emph{The cubic Szeg\H{o} equation}, Ann. Sci. l'\'Ec. Norm. Sup\'er. (4) 43 (2010), 761-810\label{gerardgrellier1}
\bibitem{GerardGrellier0}Gérard, P.; Grellier, S. \emph{Effective integrable dynamics for a certain nonlinear wave equation}, Anal. PDE, 5(2012), 1139-1155\label{gerardgrellier effective dynamic}
 
\bibitem{GerardGrellier2}Gérard, P.; Grellier, S. \emph{Invariant tori for the cubic Szeg\H{o} equation}, Invent. Math. 187:3(2012),707-754. MR 2944951 Zbl 06021979 \label{gerardgrellier2}
\bibitem{GerardGrellier3}Gérard, P.; Grellier, S. \emph{On the growth of Sobolev norms for the cubic Szeg\H{o} equation}, text of a talk at IHES, seminar Laurent Schwartz, January 6, 2015.  \label{gerardgrellier growth of sobolev norm}

\bibitem{GG}Gérard, P., Grellier, S., An explicit formula for the cubic Szegö equation, Trans. Amer. Math. Soc. 367 (2015), 2979-2995\label{Gerard-grellier explicit formula szego equation}
\bibitem{GG1}Gérard, P., Grellier, S., The cubic Szegő equation and
Hankel operators, volume 389 of Astérisque. Soc. Math. de France,
2017.\label{Gerard grellier book cubic szego equation and hankel operators}

\bibitem{GB}Grébert, B. \emph{Birkhoff Normal Form and Hamiltonian PDEs}, pp.1-46 in Partial differential equations and applications, edited by X.P.Wang and C.Zhong, Sémin. Congr.15, Soc.Math.France, Paris, 2007. MR2009d:37130 Zbl 1157.37019 \label{GB birkoff normal form}
\bibitem{Kappeler Grebert}Grébert, B.; Kappeler, T. \emph{The
Defocusing NLS Equation and Its Normal Form}, Series of Lectures in Mathematics, European Mathematical Society, 2014. \label{Kappeler Grebert}

\bibitem{ Grebert Thomann}Grébert, B.; Thomann, L. \emph{Resonant dynamics for the quintic nonlinear Schr\"odinger equation}, Ann. Inst. H. Poincar\'e Anal. Non Lin\'eaire 29(3) (2012) 455-477. \label{Grebert Thomann Resonant dynamics for the quintic NLS}

\bibitem{ Haus Procesi beating solution for quintic nls}Haus, E.; Procesi, M. \emph{KAM for beating solutions of the quintic NLS}, Comm. Math. Phys. $\mathbf{354}$ (2017), no.3, 1101-1132   \label{Haus Procesi beating solution for quintic nls}



\bibitem{}Lax, P. \emph{Integrals of Nonlinear Equations of
Evolution and Solitary Waves}, Comm. Pure Appl. Math. Volume $21$, 1968, Pages 467–490\label{LAX lax pair}

\bibitem{}Liouville, J. \emph{Sur des classes tr\`es--\'etendues de quantités dont la valeur n'est ni alg\'ebrique, ni m\^eme r\'eductible \`a des irrationnelles alg\'ebriques}, J.Math. pures et app. $\mathbf{16}$(1851), 133-142 \label{Liouville diophantine approximation}

\bibitem{Ogawa}Ogawa, T. \emph{A proof of Trudinger’s inequality and its application to nonlinear Schrödinger equations}, Nonlinear Anal. 14 (1990), 765–769. \label{Ogawa trudinger}

\bibitem{Peller}Peller, V. V. \emph{Hankel operators of class $\mathcal{S}_p$ and their applications (rational approximation, Gaussian processes, the problem of majorization of operators)}, Math. USSR Sb.41(1982), 443–479\label{Peller Hankel operators of class Sp}



\bibitem{Vladimirov}Vladimirov, M. V. \emph{On the solvability of a mixed problem for a nonlinear equation of Schr\"odinger type}, Sov. Math. Dokl. 29 (1984), 281–284.\label{Vladimirov trudinger}

\bibitem{Yudovich}Yudovich, V. I . \emph{Non-stationary flow of an ideal incompressible liquid}, USSR Comput. Math. Math. Phys. 3 (1963), 1407–1456 (english), Zh. Vuch. Mat. 3 (1963), 1032–
1066 (russian).\label{Yudovich trudinger}

\bibitem{ZAKHAROV SHABAT}Zakharov, V.E.,  Shabat, A.B.,
\emph{Exact theory of two-dimensional self-focusing and one-dimensional
self-modulation of waves in nonlinear media}, Soviet Physics JETP 34-62, 1972.\label{ZAKHAROV SHABAT}



\bibitem{Zhidkov qualitative theory book}Zhidkov, P. \emph{Korteweg--de Vries and nonlinear Schr\"odinger equations: qualitative theory}, Lecture Notes in Mathematics 1756, Springer-Verlag, Berlin, 2001.\label{Zhidov qualitative theory book}
\end{thebibliography}
\end{document}